\documentclass[invmat]{svjour}
\usepackage{color}
\usepackage{amsmath}
\usepackage{verbatim}
\usepackage{graphicx}
\usepackage{psfrag}
\usepackage{mathtools}

\def\bfR{\mbox{\boldmath$R$}}
\def\supp{{\rm supp}\thinspace}
\def\tsigma{{\tilde \sigma}}
\def\etsigma{{\tilde \sigma^\eps}}
\def\esigma{{\sigma^\eps}}

\def\eps{{\varepsilon}}
\def\eu{u^\eps }
\def\erho{\rho^\eps }
\def\rhom{\rho_m}

\def\rhomx{\rho_{m,x_0}}
\def\ds{\displaystyle}
\def\e{\zeta}
\numberwithin{equation}{section}
\begin{document}
\title{A geometric one-sided inequality for zero-viscosity limits}

\titlerunning{A geometric one-sided inequality}
\author{Yong-Jung Kim
%\thanks{This work was supported by the National Research Foundation of Korea (No. 2009-0077987).}
}

 \institute{Department of Mathematical Sciences, KAIST \\
291 Daehak-ro, Yuseong-gu, Daejeon 305-701, Republic of Korea\\ \\
Fax: +82-42-350-5710\\ \\
email: yongkim@kaist.edu}
\date{\today}
% The correct dates will be entered by the editor
\maketitle
\begin{abstract}
The Oleinik inequality for conservation laws and Aronson-Benilan type inequalities for porous medium or p-Laplacian equations are one-sided inequalities that provide the fundamental features of the solution such as the uniqueness and sharp regularity. In this paper such one-sided inequalities are unified and generalized for a wide class of first and second order equations in the form of
$$
u_t=\sigma(t,u,u_x,u_{xx}),\quad
u(x,0)=u^0(x)\ge0,\quad t>0,\,x\in\bfR,
$$
where the non-strict parabolicity ${\partial\over\partial q} \sigma(t,z,p,q)\ge0$ is assumed. The generalization or unification of one-sided inequalities is given in a geometric statement that the zero level set $$
A(t;m,x_0):=\{x:\rhom(x-x_0,t)-u(x,t)>0\}
$$
is connected for all $t,m>0$ and $x_0\in\bfR$, where $\rhom$ is the fundamental solution with mass $m>0$. This geometric statement is shown to be equivalent to the previously mentioned one-sided inequalities and used to obtain uniqueness and TV boundedness of conservation laws without convexity assumption. Multi-dimensional extension for the heat equation is also given.
\end{abstract}

\tableofcontents

%%%%%%%%%%%%%%%%%%%
\section{Introduction}

Studies on partial differential equations, or PDEs for brevity, are mostly focused on finding properties of PDEs within a specific discipline and on developing a technique specialized to them. However, finding a common structure over different disciplines and unifying theories from different subjects into a generalized theory is the direction that mathematics should go in. The purpose of this paper is to develop geometric arguments to combine Oleinik or Aronson-Benilan type one-sided estimates that arise from various disciplines from hyperbolic to parabolic problems. This unification of existing theories from different disciplines will provide a true generalization of such theories to a wider class of PDEs. It is clear that algebraic or analytic formulas and estimates that depend on the specific PDE cannot provide such a unified theory and we need a different approach. In this paper we will see that a geometric structure of solutions may provide an excellent alternative in doing such a unification.

The main example of this paper is the entropy solution of an initial value problem of a scalar conservation law,
\begin{equation}\label{c law}
\partial_t u+\partial_x f(u)=0,\ u(x,0)=u^0(x),\quad t>0,\ x\in\bfR.
\end{equation}
Dafermos \cite{MR0481581} and Hoff \cite{MR688972} showed that, if the flux $f$ is convex,  the entropy solution satisfies the Oleinik inequality,
\begin{equation}\label{OleinikInequality}
\partial_x f'(u)\big(=f''(u)u_x\big)\le {1\over t},\quad t>0,\ x\in\bfR,
\end{equation}
in a weak sense. This is a sharp version of a one-sided inequality obtained by Oleinik \cite{MR0094541} for a uniformly convex flux case. This inequality provides a uniqueness criterion and the sharp regularity for the admissible weak solution. However, if the flux is not convex, then the Oleinik estimate fails.

One may find a similar theory from a different discipline of PDEs, nonlinear diffusion equations,
\begin{equation}\label{NONLINEARDIFFUSION}
\partial_t u=\nabla\cdot(\phi(u,\nabla u)\nabla u)=0,\ u(x,0)=u^0(x),\quad t>0,\ x\in\bfR^n.
\end{equation}
This equation is called the porous medium equation (PME) if $\phi=qu^{q-1}$ with $q>1$, the fast diffusion equation (FDE) with $q<1$, and the heat equation if $q=1$. Aronson and B\'{e}nilan \cite{MR524760} showed that, for $q\ne1$, its solution satisfies a one-sided inequality
\begin{equation}\label{ABInequality}
\Delta \wp(u)\ge -{k\over t},\quad k:={1\over n(q-1)+2},\ \wp(u):={q\over q-1}u^{q-1}.
\end{equation}
This inequality played a key role in the development of nonlinear diffusion theory that the Oleinik inequality did. The equation (\ref{NONLINEARDIFFUSION}) is called the $p$-Laplacian equation (PLE) if $\phi=|\nabla u|^{p-2}$ with $p>1$ and its solution satisfies a similar one-sided inequality. However, these inequalities depend on the homogeneity of the function $\phi$.

The Oleinik inequality for hyperbolic conservation laws and the Aronson-Benilan type inequalities for porous medium or p-Laplacian equations are one-sided inequalities that provide key features of solutions such as the uniqueness and sharp regularity. Even though these key inequalities are from different disciplines of PDEs, they reflect the very same phenomenon. However, this kind of one-sided estimates do not hold without convexity or homogeneity assumption of the problem. Such a key estimate for a general situation has been the missing ingredient to obtain theoretical progress for a long time in related disciplines.

The purpose of this paper is to present a unified and generalized version of such one-sided inequalities for a general first or second order differential equation in the form of
\begin{equation}\label{EQN}
u_t=\sigma(t,u,u_x,u_{xx}),\quad
u(x,0)=u^0(x)\ge0,\quad t>0,\,x\in\bfR,
\end{equation}
where  the sub-indices stand for partial derivatives and the initial value $u^0$ is nonnegative and bounded. The main hypothesis on $\sigma$ is the parabolicity,
\begin{equation}\label{positivity}
0\le {\partial\over\partial q}\sigma(t,z,p,q)\le C<\infty,
\end{equation}
which is not necessarily uniformly parabolic. Here, we denote $u,u_x$ and $u_{xx}$ by $z,p$ and $q$, respectively.

The solution of (\ref{EQN}) is not unique in general. For example, the conservation law (\ref{c law}) is in this form with $\sigma(t,z,p,q)=f'(z)p$ and its weak solution is not unique. However, it is well known that the zero viscosity limit of a conservation law is the entropy solution. For the wide class of problems in (\ref{EQN}), one can still consider zero-viscosity limits as follows. Let $\eps>0$ be small and $\eu(x,t)$ be the solution to a perturbed problem
\begin{equation}\label{EQNxperturbed}
\partial_t \eu=\esigma(t,\eu,\eu_x,\eu_{xx}),\quad
\eu(x,0)=u^{\eps0}(x),\quad t>0,\,x\in\bfR,
\end{equation}
where ${\eu}^0$ and $\esigma$ are smooth perturbations of $u^0$ and $\sigma$, respectively, and $\esigma$ satisfies
\begin{equation}\label{positivity_perturbed}
\eps\le{\partial\over\partial q}\esigma(t,z,p,q)\le{\cal C}<\infty.
\end{equation}
If $\sigma$ is smooth, one may simply put $\esigma=\sigma+\eps q$. In fact $\esigma$ is required to be $C^2$ with respect to $q$, $C^1$ with respect to $p$, and $C^0$ with respect to $z$. For a conservation law case, $\sigma=f'(z)p$ is already smooth enough and such a perturbation is standard. The convergence of the perturbed problem is known for many cases including PME and PLE. The focus of this paper is the structure of the limit of $u^\eps$ as $\eps\to0$ and hence we assume such a convergence.

Let $\rhom$ be a nonnegative fundamental solution of (\ref{EQN}) with mass $m>0$, i.e.,
$$
\rhom(x,t)\to m\delta(x)\mbox{~~as~~} t\to0\mbox{~~in~~}L^1(\bfR).
$$
The idea for the unification of the one-sided inequalities comes from the observation that they are actually comparisons with fundamental solutions, where a fundamental solution satisfies the equality. For example, the Oleinik inequality can be written as $\partial_x f'(u(x,t))\le \partial_x f'(\rhom(x,t))$ for all $x\in\supp(\rhom(t))$. Similarly, the Aronson-B\'{e}nilan inequality becomes $\Delta(\wp(u))\ge\Delta(\wp(\rhom))$ for all $x\in\supp(\rhom(t))$. This observation indicates that the unified version of such one-sided inequalities should be a comparison with fundamental solutions.

The unification process is to find the basic common feature, which will be given in terms of geometric concept of connectedness of a level set. First we introduce the connectedness of the zero level set in a modified way.
\begin{definition}\label{Def.Connectedness} The zero level set $A:=\{x\in\bfR:e(x)>0\}$ is connectable by adding zeros, or simply connectable, if there exists a connected set $B$ such that $A\subset B\subset\{x\in\bfR:e(x)\ge0\}$.
\end{definition}
In other words, if we can connect the zero level set $A:=\{x\in\bfR:e(x)>0\}$ by adding a part of zeros of the function $e(x)$, we call it connectable or simply connected. For example, if the graph of $e$ is as given in Figure \ref{fig0}, its zero level set is connectable by adding zeros. In other words we are actually interested in sign changes. In this paper the connectedness the level set is always in this sense. Notice that, for a uniformly parabolic case that ${\partial\over\partial q}\sigma>\eps>0$, the usual connectedness is just enough. However, to include the case ${\partial\over\partial q}\sigma\ge0$, we need to generalize the connectedness as in the definition.
\begin{figure}
\centering
\includegraphics[width=7cm]{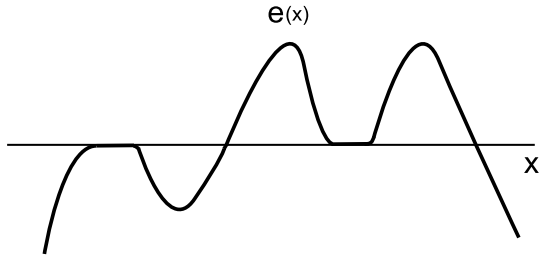}
\caption{The zero level set $A:=\{x\in\bfR:e(x)>0\}$ of the function $e(x)$ in the figure is not connected. The set  $\{x\in\bfR:e(x)\ge0\}$ is not connected neither. However, the set $A$ is connectable by adding zeros in the sense of Definition \ref{Def.Connectedness}.}\label{fig0}
\end{figure}

Finally, we are ready to present the unified version of the one-sided inequalities.
\begin{theorem}[Geometric one-sided inequality] \label{thm.generalization}
Let $u(x,t)$ be the nonnegative zero-viscosity solution of (\ref{EQN})-(\ref{positivity}), $\rhom(x,t)$ be the fundamental solution with mass $m>0$, and $e_{m,x_0}(x,t):=\rhom(x-x_0,t)-u(x,t)$. Then, the zero level set $A(t;m,x_0):=\{x\in\bfR:e_{m,x_0}(x,t)>0\}$ is connectable for all $m,t>0$ and $x_0\in\bfR$.
\end{theorem}

The proof of Theorem \ref{thm.generalization} is given in Section \ref{Sect.ProofOfMainTheorem} using the zero set theory (see \cite{MR953678,MR672070}). Main parts of this paper come after the proof. It is shown that the connectedness of the level set is equivalent to the Oleinik inequality for a conservation law with a convex flux, Theorem \ref{thm.equi1}, and the Aronson-Benilan inequality for the porous medium and the fast diffusion equation, Theorem \ref{thm.equi2}. In this way we may see that the connectedness of the level set is a true unification of the one-sided inequalities from different disciplines.

The estimates for solutions of PDEs are usually obtained by using analytical relations, but not geometrical ones. However, the geometric approach in this paper will show that they are equally useful and convenient. In fact, in certain situations, geometric relations provide simple and intuitive way to estimate solutions. One of the purposes of this paper is to develop geometric approaches to estimate solutions of PDEs. In Section \ref{Sect.Steepness}, the connectivity of the zero level set in Theorem \ref{thm.generalization} is developed to obtain geometric steepness estimates of the solution. A bounded solution is compared to fundamental solutions in Theorem \ref{thm.steepness}. The steepness comparison can be considered as an estimate of solution gradients. However, in a delicate situation as in the theorem, geometric arguments may provide a relatively simple and intuitive way to estimate of solutions which is not possible by usual analytical approaches.

The steepness comparison is used as a key to show the uniqueness of the solution to a conservation law without convexity in Theorem \ref{thm.uniquenessCLAW}. It is shown in the theorem that, if the zero level set in Theorem \ref{thm.generalization} is connected for all fundamental solutions, it is the unique entropy solution even without the convexity assumption. This steepness comparison also shows that the total variation of the solution is uniformly bounded for any given time and bounded domain in Theorem \ref{thm.TVB}. Such an estimate is well known for a convex flux case, where the Oleinik inequality is the key for the estimate. Hence it is not surprising that the geometric generalization of the Oleinik inequality gives a similar TV estimation without the convexity assumption.

The challenge of the geometric approach of this paper is to extend it to multi-dimensions. The main difficulty is that there is no multi-dimensional version of a lap number theory nor a zero set theory. In fact, the number of connected components of zero level set has no monotonicity property, which is the reason why there is no multi-dimensional version of such theories. However, our interest is a special case of comparing a general bounded solution $u(x,t)$ to a fundamental solution $\rhom(x,t)$. Such connectedness will provide the steepness comparison for multi-dimensions. In fact, it is proved in Theorem \ref{thm.HeatEqn} that the level set is convex for the heat equation. A brief discussion for an extension of theory to multi-dimensions is given in Section \ref{sect.Rn}.

%%%%%%%%%%%%%%%%%%%
\section{Proof of the geometric one-sided inequality} \label{Sect.ProofOfMainTheorem}
In this section we prove Theorem \ref{thm.generalization}. The proof depends on the monotone decrease of the lap number or the number of zero points (see \cite{MR953678,MR672070,Sturm}). For an equation in a divergence form, the lap number theory is more convenient. Since our equation (\ref{EQN}) is in a non-divergence form, the zero set theory is more convenient. The following lemma is a simplified version of Angenent \cite[Theorem B]{MR953678}.
\begin{lemma}[zero set theory] \label{lem.angenent88}
Let $e(x,t)$ be a nontrivial bounded solution to
$$
e_t=a(x,t)e_{xx}+b(x,t)e_x+c(x,t)e,\quad e(x,0)=e_0(x),
$$
where $a\ge\eps$ for some $\eps>0$ and $a,a^{-1},a_t,a_x,a_{xx},b,b_t,b_x,c$ are bounded. Then, for all $t>0$, the zeros of $e(x,t)$ are discrete and the number of zeros are decreasing as $t\to\infty$.
\end{lemma}

We will apply this lemma to the difference $e(x,t):=\rhom(x,t)-u(x,t)$ and show that the zero level set $A:=\{x:e(x,t)>0\}$ is connected (or connectable) for all $t>0$. The number of sign changes of $e(x,t)$ for a small $t>0$ is at most two since $\rhom(x,t)$ is a delta sequence as $t\to0$ and $u(x,0)$ is bounded and nonnegative. Hence what we need is a special case of the zero set theory. Notice that the zero set is just the boundary of the zero level set and hence the theory can be written in terms of the number of connected components of the zero level set. This is the idea that can be naturally extended to multi-dimensions. The following corollary is the one we need for our purpose.

\begin{corollary} \label{cor.zeroset}
Under the same assumptions in Lemma \ref{lem.angenent88}, $A(t):=\{x\in\bfR: e(x,t)>0\}$ is connected for all $t>0$ if $A(0)$ is connected.
\end{corollary}
\begin{proof}
If $A(0)=\bfR$ or $A(0)=\emptyset$, then the initial value $e_0(x)$ has no zero point. Therefore, $e(x,t)$ has no zero for all $t>0$ by the zero set theory (or by the maximum principle) and hence $A(t)$ is also connected. Suppose that $A(0)$ is a half real line. Then, $e(x,t)$ has at most one zero and hence $A(t)$ is also connected. Suppose that $A(0)$ is a bounded interval. Then, $e(x,t)$ has at most two zero points. If $e(x,t)$ has no or a single zero, then $A(t)$ is connected. Let $e(x,t)$ has two zeros, $x_1(t)<x_2(t)$. Then, the zero set theory implies that $e(x,\tau)$ has two zeros $x_1(\tau)<x_2(\tau)$ for all $0<\tau<t$. Since $e(x,t)$ has the same sign on the domain bounded by $\tau=0$, $\tau=t$, $x=x_1(\tau)$ and $x=x_2(\tau)$, $A(t)$ is an interval and hence connected.
$\hfill\qed$\end{proof}

Consider a solution $u(x,t)$ of
\begin{equation}\label{EQNx}
\partial_t u=\tsigma(x,t,u,u_x,u_{xx}),\
u(x,0)=u^0(x)\ge0,\ t>0,\,x\in\bfR,
\end{equation}
where the initial value $u^0$ is nonnegative and bounded. We consider a parabolic case that satisfies
\begin{equation}\label{positivityx}
0\le{\partial\over\partial q}\tsigma(x,t,z,p,q)\le C.
\end{equation}
In this notation, $\tsigma$ is allowed to have $x$ dependency and hence is a generalized form of (\ref{EQN})--(\ref{positivity}). For the uniqueness of the problem we consider the zero-viscosity limit as usual. Let $\eps>0$ be small and $\eu(x,t)$ be the solution to a perturbed problem
\begin{equation}\label{EQNxperturbed}
\partial_t \eu=\etsigma(x,t,\eu,\eu_x,\eu_{xx}),\
\eu(x,0)=u^{\eps0}(x),\ t>0,\,x\in\bfR,
\end{equation}
where ${\eu}^0$ and $\etsigma$ are smooth perturbations of $u^0$ and $\tsigma$, respectively, and $\etsigma$ satisfies
\begin{equation}\label{positivity_perturbed}
\eps\le{\partial\over\partial q}\etsigma(x,t,z,p,q)\le{\cal C}<\infty.
\end{equation}
Notice that the perturbed problem is a special case of the original problem. Hence the properties of the solutions of (\ref{EQNx}) hold true for solutions of perturbed problem. However, certain properties hold for the perturbed problem only, which will be discussed below.

The regularity of the solution $\eu$ of the perturbed pr:20oblem, the convergence to a weak solution $\eu\to u$ as $\eps\to0$, and the uniqueness of the limit are known for several cases such as conservation laws, porous medium equations, and $p$-Laplacian equations. We will call the limit as the zero-viscosity limit. However, there is no such a theory under the generality in (\ref{EQNx}-\ref{positivityx}). Therefore, to complete the theory for an individual equation, such a zero viscosity limit should be obtained first. The following study is about the structure of such the zero viscosity limit when it does exist.

Let $u(x,t)$ be the solution of (\ref{EQNx}) given as a zero-viscosity limit of solutions $\eu$ of the perturbed problem (\ref{EQNxperturbed}), i.e.,
$$
\eu\to u\mbox{~~~a.e.~~~as~~~}\eps\to0.
$$
The fundamental solution $\rhom(x,t)$ is also the one given as a zero viscosity limit from the same perturbation process. Hence we assume that $u^\eps$ and $\rhom^\eps$ are smooth solutions of the perturbed problem that converges to $u$ and $\rhom$ as $\eps\to0$, respectively. These limits are solutions of the original problem.

\begin{theorem}\label{thm.generalization x}
Let $u(x,t)$ be the nonnegative zero-viscosity solution of (\ref{EQNx})--(\ref{positivityx}), $\rhomx(x,t)$ be a fundamental solution with $\rhomx(x,0)=m\delta(x-x_0)$, and $e_{m,x_0}(x,t):=\rhomx(x,t)-u(x,t)$. Then, the zero level set $A(t;m,x_0):=\{x\in\bfR:e_{m,x_0}(x,t)>0\}$ is connectable for all $m,t>0$ and $x_0\in\bfR$.
\end{theorem}
\begin{proof}
Let $\eu$ be the smooth solution of (\ref{EQNxperturbed}) that converges to $u$ as $\eps\to0$. Similarly, let $\erho_{m,x_0}$ be the smooth solution that converges to $\rhomx$ as $\eps\to0$. The proof of the theorem consists of two steps. The first step is to show that the zero level set of the perturbed problem,
$$A^\eps(t;m,x_0):=\{x\in\bfR:\erho_{m,x_0}(x,t)-\eu(x,t)>0\},$$
is connected. Let $e_{m,x_0}^\eps(x,t):=\erho_{m,x_0}(x,t)-\eu(x,t)$. Then, subtracting (\ref{EQNxperturbed}) from the corresponding equation for $\erho_{m,x_0}$ gives
$$
\partial_t e^\eps_{m,x_0} =\tsigma^\eps(x,t,\erho_m,\partial_x \erho_{m,x_0},\partial^2_{x}\erho_{m,x_0}) -\tsigma^\eps(x,t,\eu,\partial_x\eu,\partial_x^2\eu).
$$
One may rewrite it as
$$
\partial_t e^\eps_{m,x_0}=a(x,t)\partial_x^2e^\eps_{m,x_0} +b(x,t)\partial_x e^\eps_{m,x_0}+c(x,t)e_{m,x_0}^\eps,
$$
where
\begin{eqnarray*}
&&a(x,t)\\
&&\ :={\tsigma^\eps(x,t,\erho_{m,x_0}, \partial_x\erho_{m,x_0},\partial_x^2\erho_{m,x_0}) -\tsigma^\eps(x,t,\erho_{m,x_0}, \partial_x\erho_{m,x_0},\eu_{xx})\over \partial_x^2\erho_{m,x_0}-\eu_{xx}},\\
&&b(x,t):={\tsigma^\eps(x,t,\erho_{m,x_0}, \partial_x\erho_{m,x_0},\eu_{xx}) -\tsigma^\eps(x,t,\erho_{m,x_0},\eu_x,\eu_{xx})\over \partial_x\erho_{m,x_0}-\eu_{x}},\\
&&c(x,t):={\tsigma^\eps(x,t,\erho_{m,x_0},\eu_x,\eu_{xx}) -\tsigma^\eps(x,t,\eu,\eu_x,\eu_{xx})\over \erho_{m,x_0}-\eu}.
\end{eqnarray*}
The regularity of $\etsigma$, the smoothness of the solutions $\eu$ and $\erho_{m,x_0}$, and the uniform parabolicity in(\ref{positivity_perturbed}) imply that $a\ge\eps,a^{-1},a_t,a_x,a_{xx},b,b_x $ and $c$ are bounded. It is clear that the number of connected components of the zero level set $A^\eps(t;m,x_0):=\{x\in\bfR:\erho_{m,x_0}(x,t)-\eu(x,t)>0\}$
is one for $t>0$ small since $\erho_{m,x_0}(x,t)$ is a delta-sequence as $t\to0$ and the initial value $u^{\eps0}(x)$ is bounded and smooth. Therefore, Corollary \ref{cor.zeroset} implies that the set $A^\eps(t;m,x_0)$ is connected for all $m,t>0$ and $x_0\in\bfR$.

Next we show that the zero level set $A(t;m,x_0):=\{x\in\bfR:e_{m,x_0}(x,t)>0\}$ is connectable. The advantage of the use of the connectability in Definition \ref{Def.Connectedness} is that such a geometric structure is preserved after the above limiting process. Suppose that $A(t;m,x_0)$ is not connectable. Then $A(t;m,x_0)$ has two disjoint components that cannot be connected by simply adding zeros of $e_{m,x_0}:=\rhomx(x,t)-u(x,t)$. In other words there is a negative point of $e_{m,x_0}$ between two components of $A(t;m,x_0)$. Therefore, there are three points $x_1<x_2<x_3$ such that $e_{m,x_0}(x_1,t)>0$, $e_{m,x_0}(x_2,t)<0$ and $e_{m,x_0}(x_3,t)>0$. Since $e_{m,x_0}^\eps\to e_{m,x_0}$ pointwise as $\eps\to0$, there exists $\eps_0>0$ such that
$e^{\eps_0}_m(x_1,t),e^{\eps_0}_{m,x_0}(x_3,t)>0$ and $e^{\eps_0}_{m,x_0}(x_2,t)<0$, i.e., $A^{\eps_0}(t;m,x_0)$ is disconnected. However, it contradicts the previous result and we may conclude that $$A(t;m,x_0):=\{x\in\bfR:e_{m,x_0}(x,t)>0\}$$ is connectable.$\hfill\qed$
\end{proof}

Theorem \ref{thm.generalization} is an immediate corollary of Theorem \ref{thm.generalization x}. All we have to show is $\rhomx(x,t)=\rhom(x-x_0,t)$.

\medskip
\noindent
{\bf Proof of Theorem \ref{thm.generalization}:}\ \
Let $u(x,t)$ be the solution of (\ref{EQN}), i.e.,
$$
\partial_t u=\sigma(t,u,u_x,u_{xx}),\quad
u(x,0)=u^0(x),\quad t>0,\,x\in\bfR,
$$
and $\rhomx$ be the fundamental solution with $\rhomx(x,0)=m\delta(x-x_0)$. Then, since the equation is autonomous with respect to the space variable, one may easily see that $\rhomx(x,t)=\rhom(x-x_0,t)$, where $\rhom(x,t)$ is the fundamental solution with $\rhom(x,0)=m\delta(x)$. Therefore, the zero level set
$$
\{x:\rhom(x-x_0,t)-u(x,t)>0\} =\{x:\rhomx(x,t)-u(x,t)>0\}
$$
is connectable for all $m,t>0$ and $x_0\in\bfR$. $\hfill\qed$

\medskip
The connectedness of the level set $A(t;m,x_0)$ has two parameters, $m$ and $x_0$. One may freely choose the size and place of the fundamental solution $\rhomx(x,t)$ using two parameters $m>0$ and $x_0$. These free parameters provide sharp estimates of a solution $u$ in terms of the fundamental solution.

Before considering the implications of Theorem \ref{thm.generalization}, we show certain uniqueness property of the perturbed problem (\ref{EQNxperturbed})--(\ref{positivity_perturbed}) using the arguments in the proof of Theorem \ref{thm.generalization x} and the zero set theory given in Lemma \ref{lem.angenent88}.
\begin{theorem}\label{thm.uniquness} Let $u^\eps$ and $v^\eps$ be smooth bounded solutions to a regularized problem,
\begin{equation}\label{eqnPerturbedNox}
\partial_t \eu=\etsigma(x,t,\eu,\eu_x,\eu_{xx}),\quad
\eps\le {\partial\over\partial q}\etsigma(t,u,p,q)\le {\cal C},
\end{equation}
where $t>0$, $x\in\bfR$, and  $\etsigma$ is smooth. Then,
\begin{enumerate}
\item If $u^\eps(x,t_0)=v^\eps(x,t_0)$ in an interval $I\subset \bfR$ for a given $t_0>0$, then $u^\eps\equiv v^\eps$ on $\bfR\times\bfR^+$.
\item If  $\etsigma$ is autonomous with respect to the space variable $x$ and $u^\eps(\cdot,t_0)$ is constant in an interval $I\subset \bfR$ for a given $t_0>0$, then $u^\eps(x,t)=\alpha(t)$, where $\alpha(t)$ is a solution of a ordinary differential equation  $\alpha'(t)=\etsigma(t,\alpha(t),0,0)$.
\end{enumerate}
\end{theorem}
\begin{proof}
Let $e^\eps=v^\eps-u^\eps$. Then, $e^\eps$ satisfies
\begin{eqnarray}\label{equation for e}
&&e^\eps_t=a(x,t)e^\eps_{xx}+b(x,t)e^\eps_x+c(x,t)e^\eps,\\ &&e(x,0)=v^\eps(x,0)-u^\eps(x,0),\nonumber
\end{eqnarray}
where the coefficients,
\begin{eqnarray*}
&&a(x,t):={\tsigma^\eps(x,t,v^\eps,v^\eps_x,v^\eps_{xx}) -\tsigma^\eps(x,t,v^\eps,v^\eps_x,\eu_{xx})\over v^\eps_{xx}-\eu_{xx}},\\
&&b(x,t):={\tsigma^\eps(x,t,v^\eps,v^\eps_x,\eu_{xx}) -\tsigma^\eps(x,t,v^\eps,\eu_x,\eu_{xx})\over v_x^\eps-\eu_x},\\
&&c(x,t):={\tsigma^\eps(x,t,v^\eps,\eu_x,\eu_{xx}) -\tsigma^\eps(x,t,\eu,\eu_x,\eu_{xx})\over v^\eps-\eu},
\end{eqnarray*}
satisfy the conditions in Lemma \ref{lem.angenent88}. If $u^\eps(x,t_0)=v^\eps(x,t_0)$ in an interval $I\subset \bfR$ for a given $t_0>0$, then $e^\eps$ should be a trivial one since the zero set of $e^\eps(\cdot,t_0)$ is not discrete. Therefore, $v^\eps\equiv u^\eps$ and the first part of the theorem is obtained.

For the second part of the theorem, we suppose that $u^\eps(x,t_0)$ is constant for $x\in[a,b]=I$. Consider an ordinary differential equation
$$
\alpha'(t)=\etsigma(t,\alpha(t),0,0), \quad\alpha(t_0)=u(a,t_0)\in\bfR.
$$
Since a smooth perturbation $\etsigma(t,z,p,q)$ is assumed, ${\partial\etsigma(t,z,0,0)\over\partial z}$ is continuous and the classical ordinary differential equation theory gives a unique solution for all $t\ge0$. Clearly, $v(x,t)=\alpha(t)$ is a solution of (\ref{eqnPerturbedNox}), which agrees with $u$ on $I\times{t_0}$. Therefore, the first part of the theorem implies that $u(x,t)=\alpha(t)$ from the beginning.
$\hfill\qed$\end{proof}
Note that the theorem does not hold without the uniform parabolicity. The finite speed of propagation of a conservation law allows us to construct a counter example easily. For example, if two initial values agree on an interval, such an agreement persists at least certain finite time due to the finite speed of propagation. Therefore, the support of a fundamental solution $\rhom(x,t)$ is not the whole real line for a given $t>0$ in general. However, under the uniform parabolicity of the perturbed problem, the theorem gives the well-known phenomenon that the support of the solution is the whole real line, i.e., $\supp(\erho_m)=\bfR$. As a result we have the following lemma.
\begin{lemma}\label{lem.inm}
Let $\erho_m(x,t)$ be the fundamental solution of (\ref{eqnPerturbedNox}). If $m_1<m_2$, then $\erho_{m_1}(x,t)<\erho_{m_2}(x,t)$ for all $x\in\bfR$ and $t>0$. Furthermore, for any $m_0>0$, $x\in\bfR$ and $t>0$ fixed, $\rho^\eps_m(x,t)\to\rho^\eps_{m_0}(x,t)$ as $m\to m_0$.
\end{lemma}
\begin{proof} Let $e^\eps(x,t)=\erho_{m_2}(x,t)-\erho_{m_1}(x,t)$ with $m_1<m_2$. Then $e^\eps$ satisfies (\ref{equation for e}) which is uniformly parabolic with an initial value $(m_2-m_1)\delta$. Hence the solution becomes strictly positive for all $x\in\bfR$ and $t>0$. Therefore, $\erho_{m_1}(x,t)<\erho_{m_2}(x,t)$ for all $x\in\bfR$ and $t>0$, and
$$
\|\erho_{m_1}(t)-\erho_{m_2}(t)\|_{L^1}=|m_1-m_2|.
$$
Since the problem is uniformly parabolic, the fundamental solution $\rho_m^\eps(x,t)$ is continuous for all $t>0$. Hence the $L^1$ convergence implies the point-wise convergence and the proof is complete.
$\hfill\qed$\end{proof}
The lemma holds true for a perturbed problem which is uniformly parabolic. One may expect a non-strict inequality  $\rho_{m_1}(x,t)\le\rho_{m_2}(x,t)$ for $m_1<m_2$ without the uniform parabolicity. The point-wise convergence $\rho_m(x,t)\to\rho_{m_0}(x,t)$ as $m\to m_0$ may fail if $\rho_{m_0}$ is discontinuous at the given point.

Theorem \ref{thm.generalization} is about a comparison between $u(x,t)$ and $\rhom(x-x_0,t)$. Since the fundamental solution itself is also a bounded solution for all given $t>0$, one may compare two fundamental solutions using the theorem. We first obtain the shape of the fundamental solution of (\ref{EQN}) by comparing it to its space translation. The following corollary says that the fundamental solution $\rhom(x,t)$ changes its monotonicity only once.
\begin{corollary}[Fundamental solutions have no wrinkles] \label{cor.mono} Let $\rhom$ be the fundamental solution of (\ref{EQN}). Then there exists $\bar x=\bar x(t)\in\bfR$ such that $\rhom(\cdot,t)$ is increasing for $x<\bar x$ and decreasing for $x>\bar x$.
\end{corollary}
\begin{proof} The fundamental solution is nonnegative and $\rhom(x,t)\to0$ as $|x|\to\infty$. Therefore, $\rhom(\cdot,t)$ may have infinite or an odd number of monotonicity changes. Suppose that $\rhom(\cdot,t)$ has $2n-1$ number of monotonicity changes. Then, $A(t):=\{x\in\bfR:\rhom(x,t)-\rhom(x-x_0,t)>0\}$ should have $n$ components for $x_0>0$ small enough. Similarly, if the monotonicity of $\rhom(\cdot,t)$ is changed infinitely many times, then the set $A(t)$ is still disconnected for $x_0$ small enough. Therefore, Theorem \ref{thm.generalization} implies that $n=1$ and hence $\rhom(\cdot,t)$ changes its monotonicity only once.
$\hfill\qed$\end{proof}

\begin{lemma}\label{lem.inx_0}
Let $\erho_m(x,t)$ be the fundamental solution of the regularized problem (\ref{eqnPerturbedNox}). If $\bar x=\bar x(t)$ is the maximum point of $\erho_m(\cdot,t)$, then $\erho_m(\cdot,t)$ is strictly increasing on $(-\infty,\bar x)$ and strictly decreasing on $(\bar x,\infty)$.
\end{lemma}
\begin{proof}
Suppose that the monotonicity of $\erho_m(x,t)$ given in Corollary \ref{cor.mono} is not strict on $x<\bar x$. Then, there exist $a<b<\bar x$ such that $\erho_m(a,t)=\erho_m(b,t)$ and hence $\erho_m(x,t)$ is constant on the interval $[a,b]$. Theorem \ref{thm.uniquness} implies that $\erho_m(\cdot,t)=\alpha(t)$, which can not be a delta-sequence as $t\to0$. Therefore, the monotonicity of $\erho_m(\cdot,t)$ is strict on  $(-\infty,\bar x)$. Similarly, the fundamental solution is strictly decreasing on $(\bar x,\infty)$.
$\hfill\qed$\end{proof}

\begin{remark}
The strict monotonicity of the fundamental solution $\rhom(\cdot,t)$ in Lemmas \ref{lem.inm} and \ref{lem.inx_0} is not expected for the general case (\ref{EQN})--(\ref{positivity}). The fundamental solutions of the hyperbolic conservation law in Section \ref{sect.C-law}, (\ref{N-wave}), provide such examples.
\end{remark}
%%%%%%%%%%%%%%%%%%%
\section{Steepness as a geometric interpretation} \label{Sect.Steepness}

The Oleinik or the Aronson-B\'{e}nilan one-sided inequalities have another geometric interpretation that fundamental solutions are steeper than any other bounded solutions. The purpose of this section is to show that the connectedness of the level set given in Theorem \ref{thm.generalization} provides the same steepness comparison for the general case. This steepness comparison can be considered as a geometric version of estimates of solutions gradient.

First we remind and introduce notations. Let $u(x,t)$ be a bounded solution of (\ref{EQN}) and $\rhom(x,t)$ be the fundamental solution of mass $m>0$. The steepness of solution $u$ at a point $x=x_1$ is compared to the one of the fundamental solution $\rho_m$ at the point $x=x_2$ with the same value, i.e.,
$$
u(x_1,t)=\rhom(x_2,t),
$$
and with the same monotonicity. The existence and the uniqueness of such a point is from Lemma \ref{lem.inx_0} if the problem is uniformly parabolic.
Then, by letting
$$
\rho_{m,x_0}(x,t):=\rhom(x-x_0,t)\mbox{~~~with~~~}x_0:=x_1-x_2,
$$
we have $\rhomx(x_1,t)=\rhom(x_2,t)$, i.e., the graphs of $u(x,t)$ intersects the graph of $\rhomx(x,t)$ at $x=x_1$. However, if the problem is not uniformly parabolic, one need to state a little bit more generally due to non-uniqueness and possible appearance of discontinuities. Hence, at an intersection point, we may say
\begin{eqnarray}
&&[\min u(x_1\pm,t),\max u(x_1\pm,t)]\cap [\min \rhomx(x_1\pm,t),\max \rhomx(x_1\pm,t)],\nonumber\\
&&\ne\emptyset,
\label{intersection}
\end{eqnarray}
where $\min v(x\pm,t)$ and $\max v(x\pm,t)$ respectively denote the minimum and maximum of the left and right hand limit for given time $t$ and point $x$. Of course, if $u$ and $\rhom$ are continuous, then (\ref{intersection}) implies that
$$u(x_1,t)=\rhomx(x_1,t)$$
and the arguments in the following proof become simpler.

In the rest of this section we let $[a,b]$ be the maximal interval including $x_1$ such that the relation (\ref{intersection}) is satisfied. We employ the notational convention that $[a,b]:=\{a\}$ if $a=b$. Note that Theorem \ref{thm.uniquness} implies that $a=b=x_1$ for perturbed problems. However, it is possible that $a\ne b$ for a problem without uniform parabolicity, where an invicid conservation law is a good example. There are four possible scenarios of intersecting two graphs (see Figure \ref{fig.fourcases}). When $\rhomx$ and $u$ are discontinuous at $x=x_1$, the corresponding four scenarios are in Figure \ref{fig.fourcasesShock}. In the figures, only the cases that $\rhomx$ and $u$ increase at the intersection point are given. One may obviously figure out the other cases that $u$ and $\rhomx$ decrease.
\begin{figure}
\centering
\begin{minipage}[t]{5cm}
\centering \psfrag{G.S.}{G}
\includegraphics[width=3.5cm]{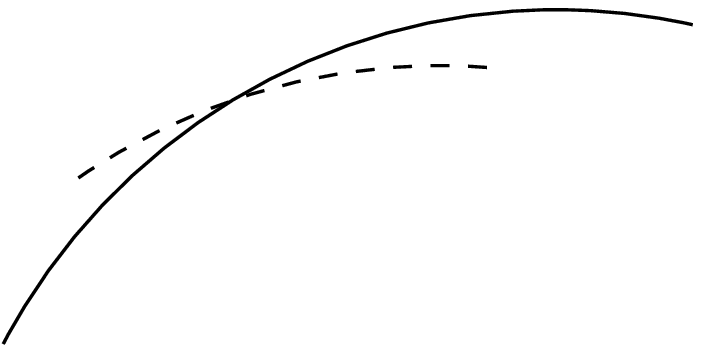}

\medskip
(a) allowed
\end{minipage}
\begin{minipage}[t]{5cm}
\centering \psfrag{II}{II}
\includegraphics[width=3.5cm]{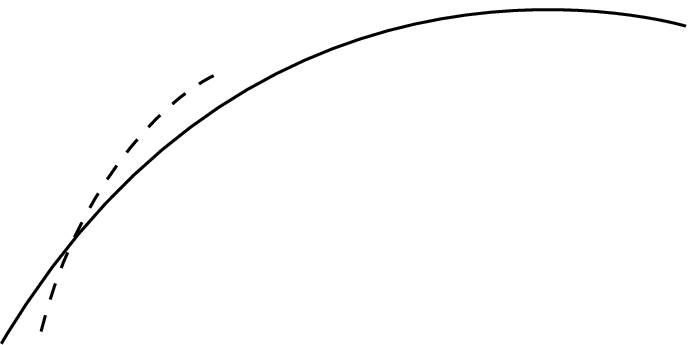}

\medskip
(b) not allowed for large $m$
\end{minipage}

\bigskip
\begin{minipage}[t]{5cm}
\centering \psfrag{R}{R}
\includegraphics[width=3.5cm]{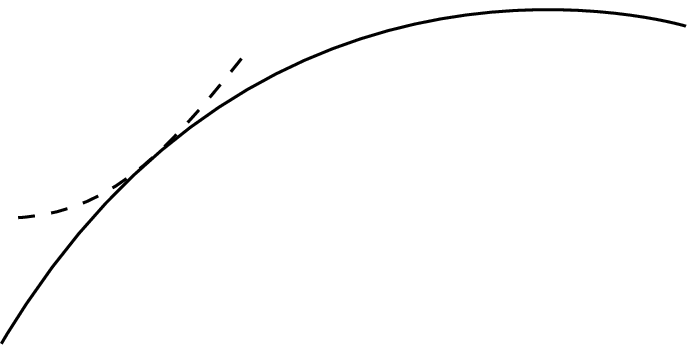}

\medskip
(c) not allowed for large $m$
\end{minipage}
\begin{minipage}[t]{5cm}
\centering \psfrag{L}{L}
\includegraphics[width=3.5cm]{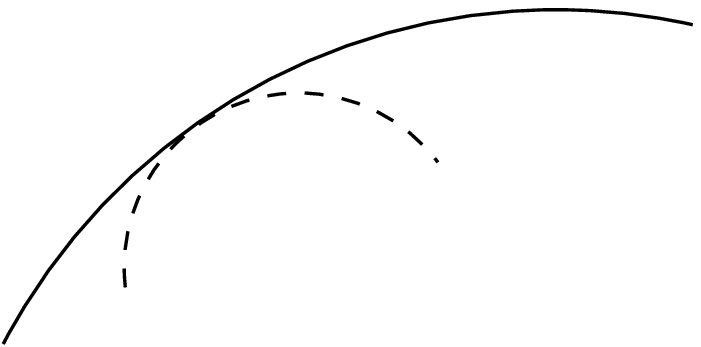}

\medskip
(d) never allowed
\end{minipage}
\caption{Four possible scenarios at the intersection point when the solutions are continuous. Solid lines are graphs of $\rhomx(\cdot,t)$ and dotted ones are of $u(\cdot,t)$.}\label{fig.fourcases}
\begin{minipage}[t]{5cm}
\centering \psfrag{G.S.}{G}
\includegraphics[width=3.5cm]{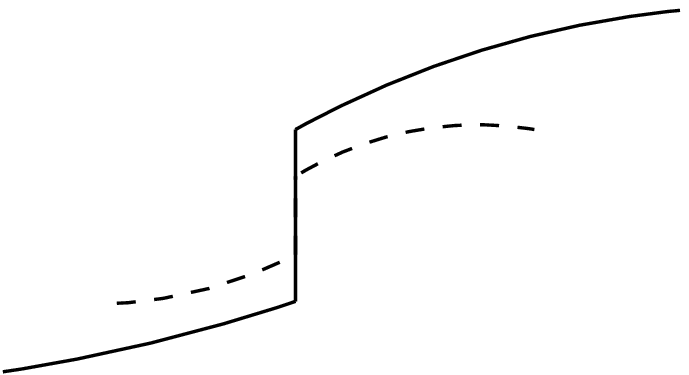}

\medskip
(a) allowed
\end{minipage}
\begin{minipage}[t]{5cm}
\centering \psfrag{II}{II}
\includegraphics[width=3.5cm]{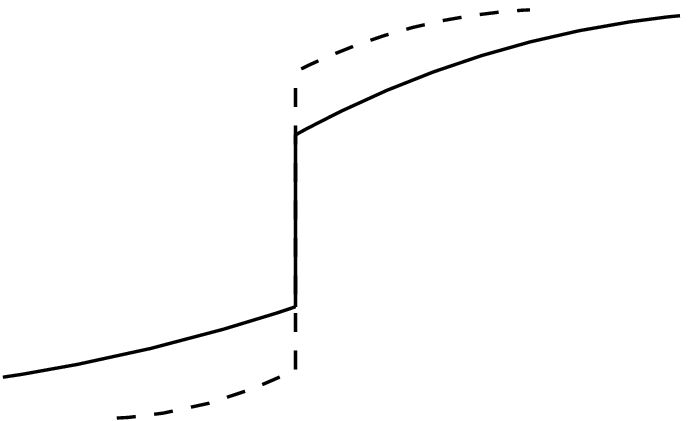}

\medskip
(b) not allowed for large $m$
\end{minipage}

\bigskip
\begin{minipage}[t]{5cm}
\centering \psfrag{R}{R}
\includegraphics[width=3.5cm]{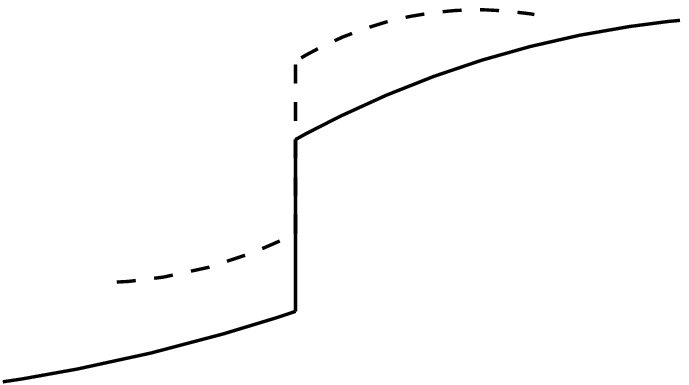}

\medskip
(c) not allowed for large $m$
\end{minipage}
\begin{minipage}[t]{5cm}
\centering \psfrag{L}{L}
\includegraphics[width=3.5cm]{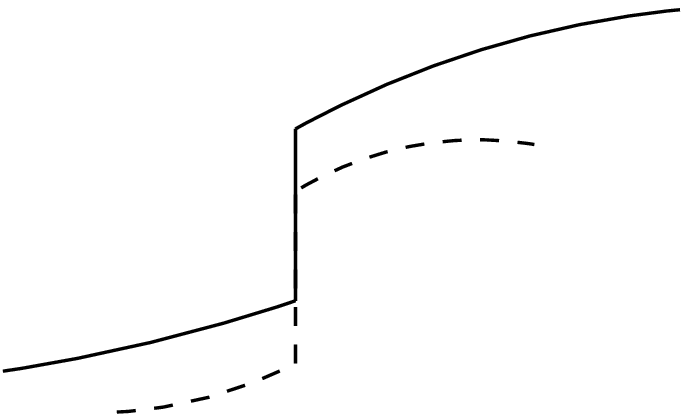}

\medskip
(d) never allowed
\end{minipage}
\caption{Four possible scenarios at the intersection point when the solutions are discontinuous. Solid lines are graphs of $\rhomx(\cdot,t)$ and dotted ones are of $u(\cdot,t)$.}\label{fig.fourcasesShock}
\end{figure}

In the rest of this section we will show which scenarios are allowed and which are not. The proofs are solely based on the connectedness of the level set in Theorem \ref{thm.generalization} and are good examples that explain how to use geometric arguments instead of analytic estimates. The proof is intuitively clear. For example, if it is the case in Figure \ref{fig.fourcases}(d), then, after shifting $\rhomx$ to right a little bit, we can make the zero level set $\{x\in\bfR:\rhom(x-x_0-\epsilon,t)-u(x,t)>0\}$ disconnected. Hence, the case is never allowed. If it is the case in Figures \ref{fig.fourcases}(b) or \ref{fig.fourcases}(c) and $m$ is large enough to satisfy $\|u(t)\|_\infty\le\|\rhomx(t)\|_\infty$, then the level set becomes disconnected before or after shifting $\rhomx$ to left a little bit. Hence, these two cases are not allowed at least for $m>0$ large. In the following theorem we state and prove this observation formally.
\begin{theorem}[Fundamental solution is the steepest.]\label{thm.steepness}
Let $u(x,t)$ be a bounded solution of (\ref{EQN}), $\rhom$ be the fundamental solution of mass $m>0$, and (\ref{intersection}) be satisfied for all $a\le x_1\le b$.
\begin{enumerate}
\item Suppose that both $u(\cdot,t)$ and $\rhomx(\cdot,t)$ are nonconstant increasing functions on $(a-\eps,a)$. Then,
\begin{enumerate}
\item If there exists $\eps>0$ such that $u(x,t)>\rhomx(x,t)$ on $(a-\eps,a)$ and $\rhomx(x,t)<u(x,t)$ on $(b,b+\eps)$, then $\rhom(x,t)\le u(x,t)$ for all $x>b$. \label{case1a}
\item If there exists $\eps>0$ such that $u(x,t)<\rhomx(x,t)$ on $(a-\eps,a)$, then $\rhom(x,t)\le u(x,t)$ for all $x>b$.\label{case1b}
\end{enumerate}
\item Suppose that both $u(\cdot,t)$ and $\rhomx(\cdot,t)$ are nonconstant decreasing functions on $(b,b+\eps)$. Then,
\begin{enumerate}
\item If there exists $\eps>0$ such that $u(x,t)>\rhomx(x,t)$ on $(a-\eps,a)$ and $\rhomx(x,t)<u(x,t)$ on $(b,b+\eps)$, then $\rhom(x,t)\le u(x,t)$ for all $x<a$.
\item If there exists $\eps>0$ such that $u(x,t)<\rhomx(x,t)$ on $(b,b+\eps)$, then $\rhom(x,t)\le u(x,t)$ for all $x<a$.
\end{enumerate}
\end{enumerate}
\end{theorem}
\begin{proof} The second part is of the dual statement of the first one and we show the first part only. We may assume without loss that both $u$ and $\rhomx$ strictly increase on $(a-\eps,a)$ after rearranging $x_0$ if needed. (This step is not needed for the perturbed problem due to Lemma \ref{lem.inx_0}.) To show (\ref{case1a}),   we assume that there exists $\alpha>b$ such that $\rhomx(\alpha,t)>u(\alpha,t)$ and derive a contradiction. Remind that $\rhomx(b+\eps,t)<u(b+\eps,t)$. We may assume $u(\cdot,t)$ and $\rhomx(\cdot,t)$ are continuous at $\alpha$ and $b+\eps$ by rearranging $\alpha$ and $\eps$ if needed. Then, the continuity of $\rhomx$ and $u$ at $\alpha$ and $b+\eps$ implies that there exists small $0<\tau<\eps$ such that $\rhomx(\alpha+\tau,t) > u(\alpha,t)$ and $u(b+\eps,t)>\rhomx(b+\eps+\tau,t)$. Therefore, the zero level set $A:=\{x\in\bfR:e(x,t)>0\}$ with $e(x,t):=\rho_{m,x_0-\tau}(x,t)-u(x,t)$ is not connectable since $e(a,t)>0$, $e(\alpha,t)>0$ and $e(b+\eps,t)<0$, which contradicts Theorem \ref{thm.generalization}. Therefore, there is no such $\alpha>b$ and hence $\rhom(x,t)\le u(x,t)$ for all $x>b$.

The proof of (\ref{case1b}) is similar to (\ref{case1a}). The difference is in the comparing points. We similarly suppose that there exists $\alpha>b$ such that $\rhomx(\alpha,t)>u(\alpha,t)$. Remind that $\rhomx(a-\eps,t)>u(a-\eps,t)$. We assume $u(\cdot,t)$ and $\rhomx(\cdot,t)$ are continuous at $\alpha$ and $a-\eps$ by rearranging $\alpha$ and $\eps$ if needed. Then, the continuity of $\rhomx$ and $u$ at $\alpha$ and $a-\eps$ implies that there exists small $0<\tau<\eps$ such that $\rhomx(\alpha+\tau,t) > u(\alpha,t)$ and $u(a-\eps,t)<\rhomx(a-\eps+\tau,t)$. We also have  $u(a,t)>\rhomx(a+\tau,t)$ Therefore, the zero level set $A:=\{x\in\bfR:e(x,t)>0\}$ with $e(x,t):=\rho_{m,x_0-\tau}(x,t)-u(x,t)$ is not connectable since $e(a,t)<0$, $e(\alpha,t)>0$ and $e(a-\eps,t)>0$, which contradicts Theorem \ref{thm.generalization}. Therefore, there is no such $\alpha>b$ and hence $\rhom(x,t)\le u(x,t)$ for all $x>b$.
$\hfill\qed$
\end{proof}

The previous theorem compares the steepness of a general bounded solution $u$ to the fundamental solution $\rho_m$ and one may obtain information or estimates of $u$ from a fundamental solution. For example, if the fundamental solution is continuous, then the general solution should be continuous. If not, one can easily construct a situation such as Figure \ref{fig.fourcases}(b) which violates Theorem \ref{thm.steepness}. If the fundamental solution contains decreasing discontinuities only, we can say that the increasing discontinuity of a weak solution is not admissible. The entropy condition of hyperbolic conservation laws is exactly the case. In certain cases, the fundamental solution is given explicitly and hence corresponding one-sided inequality is explicit. Oleinik and Aronson-B\'{e}nilan type inequalities are such examples. However, even if there is no such explicit inequalities, these steepness comparison in Theorem \ref{thm.steepness} may provide equally useful estimates for a general solution.

\begin{remark} Theorem \ref{thm.steepness}(\ref{case1a}) handles the case in Figure \ref{fig.fourcases}(b). Since $u$ is a bounded solution, there exists $m>0$ such that $\|\rhom(t)\|_\infty>\|u(t)\|_\infty$. In that case, $u(x,t)$ can not be bigger than or equal to $\rhom(x,t)$ for all $x>a$. In other words, such a case is possible only for $m>0$ small. In Section \ref{sect.C-law}, we will see that such a case is not possible at all even for a small $m$ for a convex conservation law case. However, the case is possible for small $m$ if the convexity assumption is dropped. Theorem \ref{thm.steepness}(\ref{case1b}) handles the cases in Figures \ref{fig.fourcases}(c) and \ref{fig.fourcases}(d). First, the case \ref{fig.fourcases}(d) is excluded completely. The other case \ref{fig.fourcases}(c) can be possible for $m>0$ large.
\end{remark}

\begin{remark}
The theorem does not exclude the case in Figure \ref{fig.fourcases}(a) which is usually the case if not always. This relation shows that the the fundamental solution $\rho_m$ is steeper than the general solution $u$ and such a comparison should be between two points of the same value. If the graph of the solution $u$ can touch the graph of the fundamental solution $\rho_m$ as in Figure \ref{fig.fourcases}(d), it implies that $u$ is more concave than the fundamental solution $\rho_m$ is. However, such a case is excluded and hence we may say that the fundamental solution is more concave than any other solution, which is another interpretation of the steepness.
\end{remark}

%%%%%%%%%%%%%%%%%%%
\section{Scalar conservation laws} \label{sect.C-law}

In this section we consider a scalar conservation law with a smooth flux,
\begin{equation}\label{c law4}
\partial_t u+\partial_xf(u)=0,\ u(x,0)=u^0(x)\ge0,\quad t>0,\ x\in\bfR.
\end{equation}
The flux $f$ is assumed without loss to satisfy
\begin{equation}\label{hypoC}
f(0)=f'(0)=0.
\end{equation}
This conservation law is in the form of (\ref{EQN}) with $\sigma(x,z,p,q)=-f'(z)p$, where (\ref{positivity}) is satisfied with $\partial_q\sigma=0$.

The scalar conservation law serves us for two purposes. Its solution gives a concrete example to review the steepness theory developed in the previous section. The fundamental solution of a conservation law has a rich structure and is an excellent prototype of a general case. This nonlinear hyperbolic equation is also used to show that the theory of this paper is more or less optimal and one can not expect more than the theory under the generality in this paper.

The dynamics of solutions to the conservation law is well understood if the flux is convex. However, for the general case without convexity assumption, the theory is limited even for a scalar equation case. The main obstacle to develop a theory without convexity assumption is that the Oleinik inequality does not hold for the case. H owever, the geometric version of such one-sided inequalities obtained in this paper holds true. We will apply it to hyperbolic conservation laws without convexity assumption and show that the solution with connectable zero level set is unique and is the entropy solution. We will also apply the the theory to obtain a TV boundedness of a solution without the convexity assumption. This indicates that the connectivity of the zero level set is the true generalization of the Oleinik one-sided inequality.

\subsection{Structure of fundamental solutions}

The solution of an initial value problem of an autonomous linear problem is given as the convolution between the initial value and the fundamental solution. Unfortunately, there is no such a nice scenario for nonlinear problems. However, the connectedness of the zero level set given in Theorem \ref{thm.generalization} can be successfully used to obtain key estimates of a general solution by comparing it to a fundamental solution. In fact, we have obtained a steepness estimate in Section \ref{Sect.Steepness} using the connectedness of the zero level set and will obtain more of them in following sections.

In this section we survey the structure of nonnegative fundamental solution $\rhom(x,t)$ of mass $m>0$ that satisfies
\begin{equation}
\partial_t\rhom=-\partial_xf(\rhom),\ \rhom(x,0)=m\delta(x),\quad m,t>0,\ x\in\bfR.
\end{equation}
First, one may easily check that the fundamental solution satisfies
\begin{equation}\label{SimilarityInSize}
\rhom(mx,mt)=\rho_1(x,t),\quad x\in\bfR,\  t>0.
\end{equation}
This relation shows that it is enough to consider the case with $m=1$. One can also read that solutions of different sizes live in a different time scale, where the larger one lives in a slower time scale.
\begin{remark}
The similarity structure is well known for several cases including hyperbolic conservation laws. Similarity structure is a relation between the time and the space variable. For example $\rhom(x,t)$ can be obtained from its profile at $t=1$ using an invariance relation. The relation in (\ref{SimilarityInSize}) shows a different kind of similarity structure among fundamental solutions of different sizes.
\end{remark}

We first consider a convex flux that $f''(u)\ge0$ in a weak sense. Then the fundamental solution is explicitly given by
\begin{equation}\label{N-wave}
\rhom(x,t)=\left\{\begin{array}{ccc}
g(x/t)&,&\ 0<x<a_m(t),\\
0       &,& {\rm otherwise,}\\
\end{array}\right.
\end{equation}
where $g$ is called the rarefaction profile and is given by the inverse relation of the derivative of the flux, i.e.,
\begin{equation}\label{g(x)}
f'(g(x))=x.
\end{equation}
The support of the fundamental solution is given by the equal area rule
\begin{equation}
\int_0^{a_m(t)}g(x/t)dx=m
\end{equation}
(see Dafermos \cite{MR2574377}).

Since $g$ is the inverse of an increasing function $f'$, this rarefaction profile $g$ is also an increasing function. Therefore, one can clearly see that the fundamental solution $\rhom(x,t)$ has the monotonicity structure given in Corollary \ref{cor.mono} with $\bar x(t)=a_m(t)$. In particular the decreasing part of the fundamental solution is simply the single discontinuity from the maximum to zero value. However, if $f'$ has a discontinuity, then $g$ is not strictly monotone. Hence the strict monotonicity in Lemma \ref{lem.inx_0} fails in this case. Let $m_1<m_2$. Then it is clear that $\rho_{m_1}(x,t)\le\rho_{m_2}(x,t)$ and $\rho_{m_1}(x,t)=\rho_{m_2}(x,t)$ for $0<x<a_{m_1}(t)$. Hence, the strict monotonicity in Lemma \ref{lem.inm} also fails. Suppose that $f'(u)$ is constant in an interval. Then, $g'$ has a discontinuity and hence the fundamental solution may have a increasing discontinuity. Therefore, the strict monotonicity in Lemmas \ref{lem.inm} and \ref{lem.inx_0} holds for the perturbed problems only and Corollary \ref{cor.mono} is the one we may expect for a general case without the uniform parabolicity.

The steepness comparison in Section \ref{Sect.Steepness} shows that the cases in Figures \ref{fig.fourcases}(b,c) and \ref{fig.fourcasesShock}(b,c) are not allowed for $m$ large. However, we can clearly see that those cases are not allowed even for small $m$ with convexity assumption. For example, since $\rho_{m_1}(x,t)=\rho_{m_2}(x,t)$ for $0<x<\min(a_{m_1}(t),a_{m_2}(t))$, such a case is not allowed for any $m>0$ if it is not for large $m$. On the other hand, we will observed in  the rest of this section that a conservation law without the convexity assumption provides examples that such cases may happen for small $m$. We start with a brief review of the structure of the fundamental solution.

\begin{figure}
\begin{minipage}[t]{5.3cm}
\centering
\includegraphics[width=5cm]{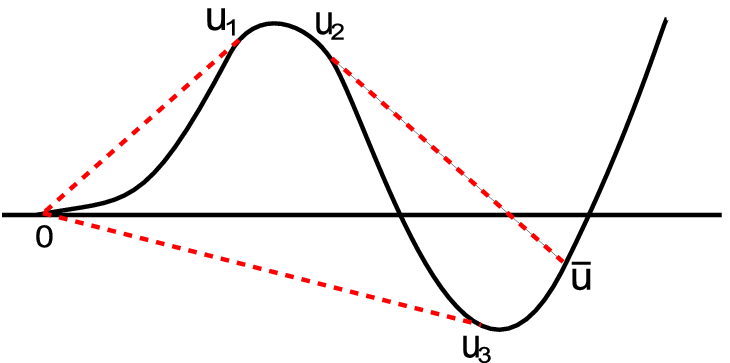}
\newline
(a) convex-concave envelops
\end{minipage}
\begin{minipage}[t]{6.1cm}
\centering
\includegraphics[width=5.2cm]{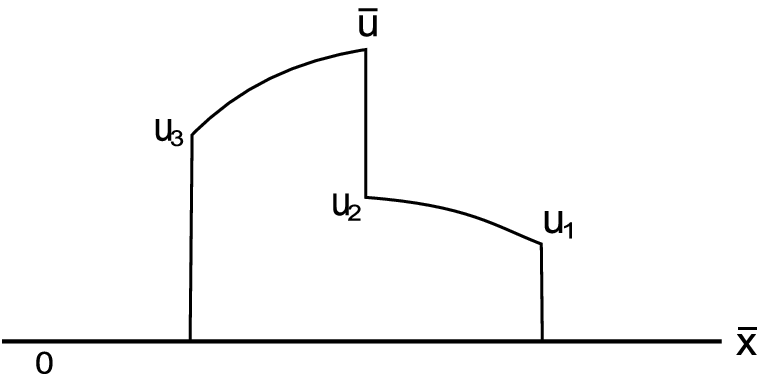}
\newline
(b) a graph of a fundamental solution
\end{minipage}
\newline
\begin{minipage}[t]{5.3cm}
\centering
\includegraphics[width=5cm]{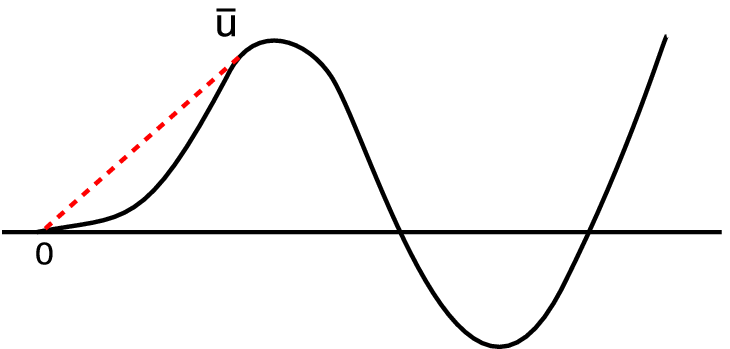}
\newline
(c) convex-concave envelops
\end{minipage}
\begin{minipage}[t]{6.1cm}
\centering
\includegraphics[width=5.2cm]{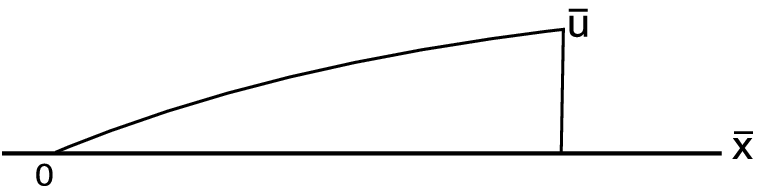}
\newline
(d)  a fundamental solution at a later time
\end{minipage}
\caption{Envelopes and corresponding fundamental solution}\label{fig.envelopes}
\end{figure}

The explicit formula (\ref{N-wave}) is valid only with convexity assumption. The fundamental solution without it is given in \cite{HaKim,KimLee}. We will briefly review its structure to use as an example to view the general theory. Using the convex-concave envelopes of the flux, one may find the left and the right side limit of a discontinuity of a fundamental solution, where the maximum of the fundamental solution is used as a parameter. Let $h(u;\bar u)$ be the lower convex envelope of $f$ on the interval $[0,\bar u]$, which is the supremum of convex functions $\eta$ such that $\eta(u)\le f(u)$ on the interval. This envelope is piecewise linear or identical to $f(u)$ (see Figure \ref{fig.envelopes}(a)). It is shown in \cite{HaKim} that, if the convex envelope $h(u;\bar u)$ has a linear part that connects two values, say $0$ and $u_3$ as in Figure \ref{fig.envelopes}(a), then the fundamental solution has a increasing discontinuity that connects $0$ and $u_3$, as in Figure \ref{fig.envelopes}(b), at the moment when $\bar u$ is the maximum of the fundamental solution $\rhom(\cdot,t)$.

The upper concave envelope $k(u;\bar u)$ is the infimum of the concave functions such that $\eta(u)\ge f(u)$. Similarly, if the concave envelope $k(u;\bar u)$ has a linear part connecting two values, say $0$ and $u_1$ or $u_2$ and $\bar u$ as in Figure \ref{fig.envelopes}(a), then the fundamental solution has decreasing discontinuities connecting $0$ and $u_1$ or $u_2$ and $\bar u$, as in Figure \ref{fig.envelopes}(b). The exact place of the discontinuities and the profile of the continuous part depend on the dynamics of envelopes at earlier times. However, the exact size of each shock can be found from the envelope at that moment of a given maximum $\bar u>0$. At a later time, when the maximum $\bar u$ of the fundamental solution is like the one in Figure \ref{fig.envelopes}(c), the convex envelope is identical to $f$ is the concave envelop is linear. Then the fundamental solution at that moment is like the one in Figure \ref{fig.envelopes}(d).

Now we consider an example of the case in Figure \ref{fig.fourcases}(b) for $m$ large. Let Figures \ref{fig.envelopes}(b) and \ref{fig.envelopes}(d) be respectively the graphs of $\rhom(x,t_1)$ and $\rhom(x,t_2)$ with $t_1<t_2$. First rewrite the relation in (\ref{SimilarityInSize}) as
$$\rho_{ma}(ax,at)=\rhom(x,t).$$
Then, we have
$$
\rhom(x,t_2)=\rho_{mt_1/t_2}(t_1x/t_2,t_1).
$$
In other words, $\rho_{mt_1/t_2}(x,t_1)$ has the shape of Figure \ref{fig.envelopes}(d) after shrinking it in $x$ direction with a ration of $t_1/t_2$. If $\rhom(x,t_1)$ plays the role of $u(x,t_1)$ and $\rho_{mt_1/t_2}(x,t_1)$ of the comparing fundamental solution, then it will give the scenario of Figure \ref{fig.fourcases}(b). Hence such a case is really possible for a general case with a large $m$. This observation also indicates that the well-known similarity structure of fundamental solution is valid only with the convexity assumption.

\subsection{Equivalence to the Oleinik inequality}

In this section we show that the connectedness of the zero level set in Theorem \ref{thm.generalization} is equivalent to the one-sided Oleinik inequality (\ref{OleinikInequality}) which is valid only with a convex flux.

\begin{theorem}\label{thm.equi1} Let $f''(u)>0$ and $\rhom(x,t)$ be given by (\ref{N-wave}). Then a non-negative bounded function $u(x)$ satisfies the Oleinik inequality
\begin{equation}\label{OleinikInequality2}
{f'(u(x))-f'(u(y))\over x-y}\le {1\over t},\quad t>0,\ x,y\in\bfR
\end{equation}
if and only if the zero level set $$A(t;m,x_0):=\{x\in\bfR:\rhom(x-x_0,t)-u(x)>0\}$$
is connected (or  connectable) for all $x_0\in\bfR$ and $m>0$.
\end{theorem}
\begin{proof} In the following the time $t>0$ is fixed and we will drop the time variable from $\rhom$ for brevity. First, note that $\rhom(x)$ satisfies
$$
{f'(\rhom(x))-f'(\rhom(y))\over x-y} ={f'(g(x/t))-f'(g(y/t))\over x-y}={1\over t}
$$
for all $0<x,y<a_m(t)$. Since $f''>0$, $f'$ is increasing and hence we have $A(t;m,x_0)=\{x\in\bfR:f'(\rhom(x-x_0))-f'(u(x))>0\}$. Suppose that the set $A$ is not connected for some $m>0$ and $x_0\in\bfR$. After an translation of $u$, we may assume $x_0=0$. Then, there exist three points $x_1<x_2<x_3$ such that $f'(\rhom(x_1))>f'(u(x_1))$, $f'(\rhom(x_2))<f'(u(x_2))$ and $f'(\rhom(x_3))>f'(u(x_3))$. Therefore, $x_1,x_2\in\supp(\rhom(t))$ and
$$
{f'(u(x_1))-f'(u(x_2))\over x_1-x_2}>{f'(\rhom(x_1))-f'(\rhom(x_2))\over x_1-x_2}={1\over t}.
$$
Hence the Oleinik inequality fails.

Now suppose that the Oleinik inequality fails. Then, there exist $x_1<x_2$ such that
$$
{f'(u(x_2))-f'(u(x_1))\over x_2-x_1}>{1\over t}.
$$
Let
$$x_0:=(x_1+x_2)/2-t[f'(u(x_1))+f'(u(x_2))]/2$$
and $m$ be so large that $a_m(t)>t\,\sup(f'(u(x)))$. Then, for $x_3:=x_0+a_m(t)-\epsilon$ with a small $\epsilon>0$, we have
\begin{eqnarray*}
&&f'(\rhom(x_1-x_0))-f'(u(x_1)) ={f'(u(x_2))-f'(u(x_1))\over2}-{x_2-x_1\over 2t}>0,\\
&&f'(\rhom(x_2-x_0))-f'(u(x_2)) ={x_2-x_1\over 2t}-{f'(u(x_2))-f'(u(x_1))\over2}<0,\\
&&f'(\rhom(x_3-x_0))-f'(u(x_3))=(a_m(t)-\epsilon)/t-f'(u(x_3))>0.
\end{eqnarray*}
In other words the zero level set $A(t;m,x_0)$ is disconnected.
$\hfill\qed$\end{proof}

Notice that the function $u$ is not necessarily a solution of the conservation law for the equivalence relation in the theorem. The time variable $t$ in the inequality (\ref{OleinikInequality2}) is related to the fundamental solution $\rhom(x,t)$ only.

\subsection{Uniqueness without convexity}

In this section we consider a conservation law (\ref{c law4}) with a nonconvex flux. Let $u(x,t)$ be a nonnegative bounded solution and have a discontinuity at a point $x_0$ such that $\ds\lim_{x\to x_0+}u(x,t)=u_r$ and $\ds\lim_{x\to x_0-}u(x,t)=u_l$. For an illustration, consider the graph of nonconvex flux given in Figure \ref{fig_walk}.
\begin{figure}
\centering
\includegraphics[width=9cm]{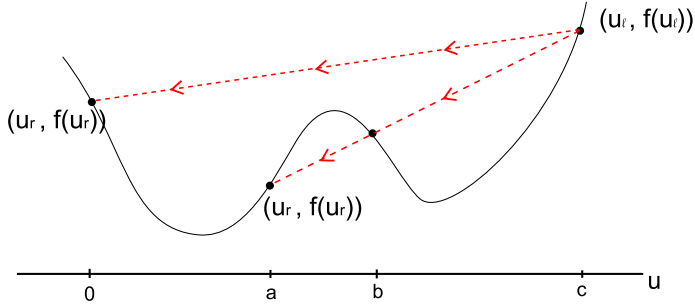}
\caption{An illustration to explain the Oleinik entropy condition. For a simpler illustration we have returned to the original case without the assumption in (\ref{hypoC}). }\label{fig_walk}
\end{figure}
Suppose that you are moving from the left limit $(u_l,f(u_l))$ to the right limit $(u_r,f(u_r))$ along the line connecting the two points. If the graph of the flux $f(u)$ lies always on your left side, then the discontinuity is admissible. For example, if the left and the right side limit pair is $(u_l,u_r)=(c,0)$ as in Figure \ref{fig_walk}, then the discontinuity is admissible. However, if $(u_l,u_r)=(c,a)$ as in Figure \ref{fig_walk}, then the graph of the flux $f(u)$ is on your right side for $a<u<b$ and hence the discontinuity is not admissible. This admissibility criterion is called the Oleinik entropy condition. If discontinuities of a weak solution satisfy the Oleinik entropy condition, then the weak solution is called the entropy solution. It is well known that the entropy solution is unique and identical to the zero-viscosity limit of its perturbed problem.

Suppose that the zero level set
\begin{equation}\label{Atmx0}
A(t;m,x_0):=\{x\in\bfR:\rhom(x-x_0,t)- u(x,t)>0\}
\end{equation}
is connected for all $t,m>0$ and $x_0$. In this section we will show that such a weak solution is the entropy solution if the flux has a single inflection point. However, for a general nonconvex flux, it can be a non-entropy solution. For example, let $u$ have a discontinuity. Then, the steepness comparison in the previous section implies that for $m>0$ sufficiently large, $\rhom(x,t)$ should have a larger discontinuity of the same monotonicity since the case in Figure \ref{fig.fourcasesShock}(a) is only the possible one for $m$ large. Of course, discontinuities of the fundamental solution are admissible ones since they are given by convex-concave envelopes. Therefore, if a flux has a property that a jump smaller than an admissible one with same monotonicity is always admissible, then $u$ should be the entropy solution. For example, if the flux has a single inflection point, one can easily check that it is the case.

For a general case, the story is quite different. For example, if the flux is given as in Figure \ref{fig_walk}, the discontinuity $(u_l,u_r)=(c,a)$ is not admissible even though a larger one $(u_l,u_r)=(c,0)$ is admissible. Furthermore, one may easily check that
\begin{equation}\label{CounterExample}
u(x,t)=\left\{\begin{array}{cc} c, \quad&x<\sigma t,\\a, \quad&x>\sigma t,\\ \end{array}\right.\quad \sigma={f(c)-f(a)\over c-a},
\end{equation}
is a weak solution solution that makes the set $A(t;m,x_0)$ be connected for all $t,m,x_0$. Unfortunately, this is not an entropy solution and hence the connectedness of the zero level set is not enough to single out the entropy solution. However,  we have the following lemma which gives a clue to obtain the uniqueness.
\begin{lemma}\label{lemma.0shock} Let $u(x)$ be a nonnegative bounded function and the zero level set $A(t;m,x_0)$ in (\ref{Atmx0}) be connected for all $t,m>0$ and $x_0\in\bfR$. Then any discontinuity of $u$ that connects $u=0$ is admissible.
\end{lemma}
\begin{proof}
Let $u(x)$ have a discontinuity at $x=x_1$ and the left side limit is $u_l>0$ and the right side limit is $u_r=0$. Suppose that the discontinuity is not admissible. Then, since a part of the graph of the flux is above the line connecting $(0,0)$ and $(u_l,f(u_l))$, the concave envelope of the flux on the interval $(0,u_l+\eps_0)$ is not a line for a small $\eps_0>0$. Let $\rhom(x,t)$ be the fundamental solution with the maximum $u_l+\eps_0$ at time $t>0$ and $\bar x$ be the maximum point. Let $x_0=x_1-\bar x- \eps_1$. Then, it is clear that the set $A(t;m,x_0)$ becomes disconnected for a small $\eps_1>0$. A diagram that shows the relation is given in Figure \ref{fig.admissibility}.
\begin{figure}
\begin{minipage}[t]{5.3cm}
\centering
\includegraphics[width=5cm]{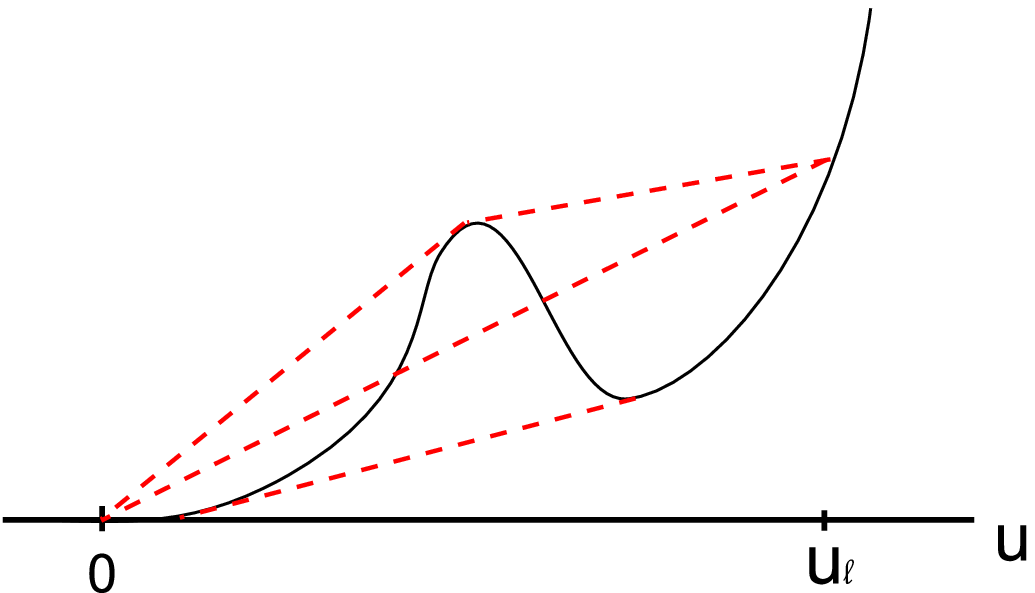}
\newline
(a) convex-concave envelops
\end{minipage}
\begin{minipage}[t]{6.1cm}
\centering
\includegraphics[width=5.2cm]{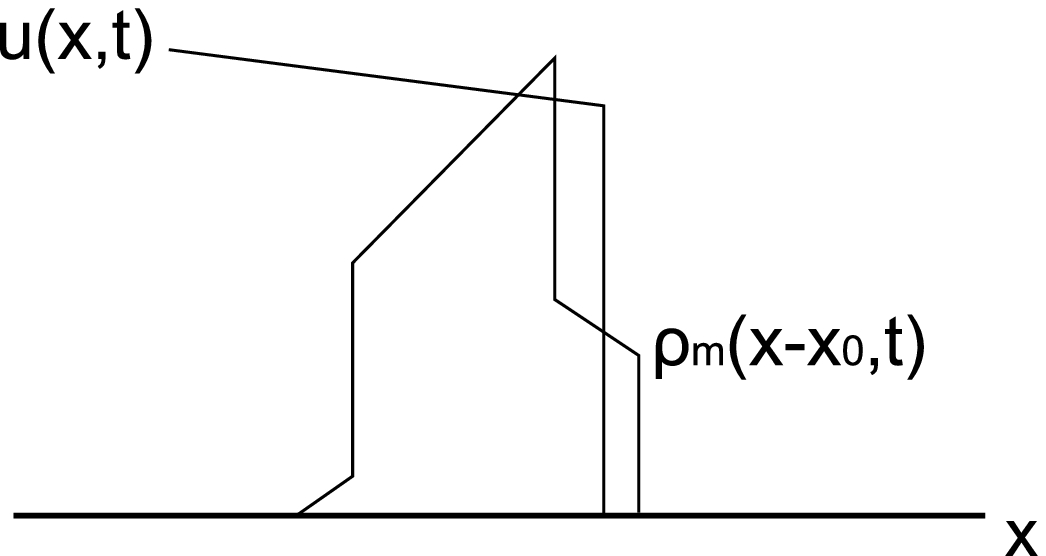}
\newline
(b) comparing $\rhom$ and $u$ near an inadmissible discontinuity
\end{minipage}
\caption{Envelopes and corresponding fundamental solution}\label{fig.admissibility}
\end{figure}
If $u_l=0$ and $u_r>0$, then one may consider the convex envelope and obtain the nonconnectedness of level set $A$ similarly.
$\hfill\qed$\end{proof}

The connectedness of the set $A(t;m,x_0)$ allows us to single out the zero-viscosity limit if the flux is convex or has a single inflection point. For a general nonconvex case, Lemma \ref{lemma.0shock} encourages us to consider fundamental solutions with a nonzero far field.

\begin{theorem}\label{thm.uniquenessCLAW}
Let $\rhom^c$ be a solution to the conservation law (\ref{c law4}) with initial value $\rhom^c(x,0)=c+m\delta(x)$, $c\ge0$, and $u(x,t)$ be a weak solution. Then, the zero level set
\begin{equation}\label{Atmx0c}
A(t;m,x_0,c):=\{x\in\bfR:\rhom^c(x-x_0,t)- u(x,t)>0\}
\end{equation}
is connected for all $t,m,c>0$ and $x_0\in\bfR$ if and only if $u(x,t)$ is the entropy solution.
\end{theorem}
\begin{proof} ($\Rightarrow$)
Suppose that $u(x,t)$ has a discontinuity that connects $c$ and $d$ with $c<d$. Then, consider the fundamental solution $\rhom^c$ which is similarly constructed using the convex and concave envelopes of the flux on the interval $[c,\bar u]$, where $\bar u$ is the maximum of the fundamental solution. This procedure is identical to the earlier case with $c=0$. Then, we may repeat the previous process of Lemma \ref{lemma.0shock} to show the admissibility of this discontinuity. The detail is omitted.

($\Leftarrow$) It is well known that the entropy solution is the zero-viscosity limit of the perturbed problem. The zero level set $A$ of a zero-viscosity limit is connected by Theorem \ref{thm.generalization} for $c=0$. For $c>0$ we may repeat the process since the zero set theory, Lemma \ref{lem.angenent88} is valid independently of $c>0$.
$\hfill\qed$\end{proof}

\begin{remark}
The connectedness of this zero level set can be used as another admissibility criterion of a conservation law without convexity. Furthermore, it gives a hope that the connectivity of the zero level set can be used for an admissibility criterion for more general problems in the form of (\ref{EQN})-(\ref{positivity}).
\end{remark}

\subsection{Boundedness of total variation}

The Oleinik inequality should be understood in a weak sense since the solution is not necessarily smooth. Hence it is preferred to write it as
\begin{equation}\label{OleinikInequality2}
{f'(u(x,t))-f'(u(y,t))\over x-y}\le {1\over t},\quad t>0,\ x,y\in\bfR.
\end{equation}
Hoff \cite{MR688972} showed that the weak solution satisfying the Oleinik inequality is unique if and only if the flux $f$ is convex. In other words, the inequality (\ref{OleinikInequality2}) does not give an uniqueness criterion without convexity of the flux. Furthermore, if the flux is not convex, the inequality is not satisfied by the entropy solution.

The theoretical development for a nonconvex case has been limited due to the lack of an Oleinik type inequality and, therefore, finding a replacement of such an inequality has been believed as a crucial step for further progress. There have been several technical developments to find the right inequality (see \cite{MR818862,MR2405854,MR1855004,MR2119939}). These efforts are related to finding a constant $C\ge0$ such that a weak version of the Oleinik inequality,
\begin{equation}\label{OleinikTV}
f'(u(x,t))-f'(u(y,t))\le {x-y\over t}+C|TV(u(0))-TV(u(t))|,
\end{equation}
is satisfied by the entropy solution. Here, $TV(u(t))$ is the total variation of the solution $u$ at a fixed time $t\ge0$. The total variation is defined by
$$
TV(u(t)):=\sup_{P}\sum_i |u(x_i,t)-u(x_{i+1},t)|,
$$
where the $\sup$ is taken over all possible partitions $P:=\{\cdots<x_i<x_{i+1}<\cdots\}$. It is clear that (\ref{OleinikTV}) is a weaker version of the Oleinik inequality (\ref{OleinikInequality2}) and that it cannot give the uniqueness since even the stronger original version does not give the uniqueness. The connectedness of the zero level set in Theorem \ref{thm.generalization} is the correct generalization that gives the uniqueness for general flux without convexity assumption, Theorem \ref{thm.uniquenessCLAW}.

The boundedness of the total variation of a solution has been one of the key estimates in the regularity theory of various problems. The one-sided Oleinik inequality actually gives TV-boundedness on any bounded domains for all $t>0$ even if it is not initially. (Notice that the inequality in (\ref{OleinikTV}) cannot be used for such a purpose since $TV(u(0))$ is already included in the estimate.) Even if there is no lower bound in the estimate the upper bound controls the variation. Roughly speaking, in terms of fundamental solution $\rhom(\cdot,t)$, the variation of the solution in the domain of size of the support of $\rhom(\cdot, t)$ is smaller than the variation of the fundamental solution due to the steepness comparison property in Theorem \ref{thm.steepness}.

Let
\begin{equation}\label{C(t)}
C(t)=\sup_{c,m>0}{2\,\sup_x(\rhom^c(x,t)-c)\over|\supp(\rhom^c-c)|}<\infty,
\end{equation}
where $\rhom^c$ is the fundamental solution in Theorem \ref{thm.uniquenessCLAW}. The variation of the fundamental solution on its support is $2\sup_x(\rhom^c(x,t)-c)$ and hence $C(t)$ is the maximum ratio of variation of all possible fundamental solutions. Therefore, one can easily see that $TV(u(t))\le C(t)|\supp(u(t))|$ since the fundamental solution is the steepest one and hence the variation of a solution $u$ in a unit interval cannot be bigger than $C(t)$. For example, for the invicid Burgers equation case, we have $C(t)={2\over t}$ and hence we have
$$TV(u(t))\le {2\over t}|\supp(u(t))|,$$
which is a way how the Oleinik one-sided inequality gives the TV-boundedness. The following theorem is a summary of the $TV$ estimate.

\begin{theorem}[TV boundedness]\label{thm.TVB} Let $u(x,t)$ be a bounded solution of (\ref{EQN}), where the flux $f$ is not necessarily convex. If $C(t)$ given by (\ref{C(t)}) is finite and $u(x,t)$ is compactly supported,  then
\begin{equation}\label{TV1}
TV(u(t))\le C(t)|\supp(u(t))|.
\end{equation}
In general, the total variation in a bounded interval $I=(a,b)$, is bounded by
\begin{equation}\label{TV2}
TV\big(u(t)|_{x\in I}\big)\le C(t)|b-a|.
\end{equation}
\end{theorem}

%\newpage%%%%%%%%%%%%%%%%
\section{Porous medium equation}

Let $u(x,t)$ be the solution to the porous medium equation
\begin{equation}\label{PME5}
\partial_t u=\Delta u^\gamma,\ u(x,0)=u^0(x)\ge0,\quad t,\gamma>0,\ x\in\bfR^n.
\end{equation}
The fundamental solution of this equation is called the Barenblatt solution and is explicitly given by
\begin{equation}\label{Barenblatt}
\rhom(x,t)=\Big (C_mt^{1-\gamma\over \gamma+1}-{\gamma-1\over2\gamma(\gamma+1)}|x|^2t^{-1}\Big)_+^{1\over \gamma-1},\quad \gamma\ne1,
\end{equation}
where we are using the notation $(f)_+:=\max(0,f)$. For $\gamma=1$, the fundamental solution is of course the Gaussian. The constant $C_m$ is positive and decided by the relation for the total mass $\int\rhom(x,t)dx=m$. For the fast diffusion regime, $0<\gamma<1$, the inside of the parenthesis is positive for all $x\in\bfR^n$ and hence $\rhom$ is strictly positive and $C^\infty$ on $\bfR^n$. It is also well studied that the general solution $u$ is also strictly positive and $C^\infty$ on $\bfR$. For the porous medium equation regime, $\gamma>1$, the fundamental solution $\rhom$ is compactly supported and $C^\infty$ in the interior of the support. The solution $u$ is also $C^\infty$ away from zero points.

For dimension $n=1$, the Aronson-B\'{e}nilan inequality in (\ref{ABInequality}) is written as
\begin{equation}\label{AB1D}
\partial_x^2\wp(u)\ge -{1\over t(\gamma+1)},\quad \wp(u):={\gamma\over \gamma-1}u^{\gamma-1},\quad \gamma\ne1,
\end{equation}
where $\wp$ is usually called pressure. One can easily check that the pressure is an increasing function for all $\gamma>0$ and the Barenblatt solution satisfies the equality in (\ref{AB1D}). In the following theorem we show that the connectedness of the zero level set in Theorem \ref{thm.generalization} is equivalent to the Aronson-B\'{e}nilan one-sided inequality.

\begin{theorem}\label{thm.equi2} Let $\rhom(x,t)$ be the Barenblatt solution with $1\ne \gamma>0$ and $u(x)$ be a non-negative bounded smooth function with possible singularity at zero points. For the case $0<\gamma<1$, $u$ is assumed to be positive. Then, the Aronson-B\'{e}nilan inequality (\ref{AB1D}) is satisfied if and only if the zero level set $A(t;m,x_0):=\{x\in\bfR:\rhom(x-x_0,t)-u(x)>0\}$ is connected (in the sense of Definition \ref{Def.Connectedness}) for all $x_0\in\bfR$ and $m>0$.
\end{theorem}
\begin{proof} ($\Rightarrow$)
Suppose that the set $A(t;m,x_0)$ is not connected for some $m>0$ and $x_0\in\bfR$. After a translation of $u(x)$, we may set $x_0=0$. Then, there exist three points $x_1<x_2<x_3$ such that $\rhom(x_1,t)>u(x_1)$, $\rhom(x_2,t)<u(x_2)$, and $\rhom(x_3,t)>u(x_3)$. Let $\e:=\wp(u)-\wp(\rhom)$ be the pressure difference. Suppose that the Aronson-B\'{e}nilan inequality (\ref{AB1D}) holds. Then,
$$
\partial_x^2\e=\partial_x^2\wp(u)-\partial_x^2\wp(\rhom)\ge0.
$$
Note that the pressure function $\wp:u\to {\gamma\over \gamma-1}u^{\gamma-1}$ is an increasing function for $\gamma>0$ and hence we have
$$
\e(x_1)<0,\ \e(x_3)<0.
$$
The maximum principle implies that $\e(x)<0$ on $(x_1,x_3)$. However, it contradicts to $\e(x_2)>0$. Hence the Aronson-B\'{e}nilan inequality should fail.

 ($\Leftarrow$)
Now suppose that there exists $x_2$ such that $\partial_x^2\wp(u(x_2))<-{1\over t(\gamma+1)}$, i.e., the Aronson-B\'{e}nilan inequality fails at a point $x_2$. Then, since $u$ is smooth away from zero points, there exist $x_1<x_2<x_3$ and $\epsilon>0$ such that $\partial_x^2\wp(u(x))<-{1\over t(\gamma+1)}-\epsilon$ and $u>0$ on $[x_1,x_3]$. Let $h(x,t)=-{1\over 2t(\gamma+1)}(x-x_0)^2+b$, where two unknowns, $x_0$ and $b$, are uniquely decided by  two relations,
$$
h(x_1,t)=\wp(u(x_1)),\quad h(x_3,t)=\wp(u(x_3)).
$$

Consider the porous medium regime $\gamma>1$ first. Then $\wp(u(x_1))>0$ and, since $h$ is not entirely negative, the constant $b$ should be positive. Set $C_m:={\gamma-1\over \gamma}bt^{\gamma-1\over \gamma+1}$. Then, $C_m>0$ and
$$
h(x,t)={\gamma\over \gamma-1}C_mt^{1-\gamma\over \gamma+1}-{1\over2(\gamma+1)}(x-x_0)^2t^{-1}.
$$
Therefore, $h(x,t)=\wp(\rhom(x-x_0,t))$ for $\rhom(x,t)>0$. Let $\e_m:=\wp(u)-\wp(\rhom)$. Then,
$$
\partial^2_{x}\e_m<-\epsilon,\ \e_m(x_1)=\e_m(x_3)=0.
$$
The strong maximum principle implies that $\e_m(x)>0$ for all $x\in(x_1,x_3)$. Therefore, there exists $m'>m$ such that $\e_{m'}(x_2)>0$. Since $\rho_{m'}(x,t)>\rhom(x,t)$ for all $x$ in the interior of the support of $\rho_{m'}(t)$, we conclude that
$$
\e_{m'}(x_1)<0,\ \e_{m'}(x_2)>0,\ \e_{m'}(x_3)<0.
$$
Therefore, the set $A(t;m',x_0)$ is disconnected.

For the fast diffusion regime, $0<\gamma<1$, we need a slightly more subtle approach to obtain the positivity of the corresponding constant $C_m>0$. Since $u$ is bounded, we may set $-B:=\wp(\sup u)<0$. Then, $\wp(u(x))\le -B$ for all $x\in\bfR$. Since $u$ is smooth, so is $\wp(u)$. Suppose that $\partial_x^2\wp(u)$ has minimum value at $x_2$ and $\partial_x^2\wp(u(x_2))<-{1\over t(\gamma+1)}$, i.e., the Aronson-B\'{e}nilan inequality fails. For a sufficiently small $\eps>0$, there exists $t_1<t$ such that $\partial_x^2\wp(u(x_2))=-{1\over t_1(\gamma+1)}-2\eps$. Then, since $u$ is smooth, there exist $x_1<x_2<x_3$ and $\epsilon>0$ such that $\partial_x^2\wp(u(x))<-{1\over t_1(\gamma+1)}-\epsilon$ for $x\in(x_1,x_3)$. Let $h^\eps(x,t):=-\big({1\over 2t_1(\gamma+1)}+\eps\big)(x-x'_0)^2+b$, where $x'_0$ and $b$ are uniquely decided by
$$
h^\eps(x_1,t)=\wp(u(x_1)),\quad h^\eps(x_3,t)=\wp(u(x_3)).
$$
Since $h^\eps(x,t)$ has the minimum curvature of $\wp(u)$ and shares the same values at $x_1$ and $x_3$ with $\wp(u)$, we have $h^\eps(x,t)\le \wp(u(x))$ for all $x\not\in(x_1,x_3)$. Since the curvature difference between $h^\eps$ and $\wp(u)$ is less than $\eps$ on the interval, we have $h^\eps(x,t)<0$ for all $x\in\bfR$. By taking smaller $\eps>0$ if needed, we obtain $h(x,t):=-{1\over 2t_1(\gamma+1)}(x-x'_0)^2+b<0$
using the same boundary condition. Therefore, $b<0$ and hence the constant $C_m:={\gamma-1\over \gamma}bt_1^{\gamma-1\over \gamma+1}$ becomes positive. The same arguments for the PME case show that there exists $m'>0$ such that $A(t_1;m',x'_0)$ is disconnected with $t_1<t$. Therefore there exists $m>0$ and $x_0$ that make $A(t;m,x_0)$ be disconnected.  $\hfill\qed$
\end{proof}

The Aronson-B\'{e}nilan one-sided inequality is valid in multi-dimensions. Hence it is natural to ask what is the corresponding equivalent concept for the multi-dimensional case. Further discussions on this matter are in the next section.

%\newpage%%%%%%%%%%%%%%%
\section{Connectivity in multi-dimensions}\label{sect.Rn}

In this section we discuss about a possibility to extend the one dimensional theory of this paper to multi-dimensions. Let $u(x,t)$ be a bounded nonnegative solution of
\begin{equation}\label{EQNRn}
\partial_t u=F(t,u,Du,D^2u),\ u(x,0)=u^0(x)\ge0,\ t>0,\ x\in\bfR^n,
\end{equation}
where the $n\times n$ matrix $D_qF(x,t,z,p,q)$ is positive definite, i.e.,
\begin{equation}\label{PositiveDefinit}
\sum_{i,j=1}^n\big(D_{q_{ij}}F(t,z,p,q)\big)\xi_i\xi_j\ge0
\end{equation}
for all $\xi_i\in\bfR$.

Remember that the one dimensional theory depends on non-increase of the number of zeros or of the lap number. An advantage of the argument in Theorem \ref{thm.generalization} in compare with the one-sided inequalities is that the connectivity is a multi-dimensional concept. Counting the number of zeros is meaningless in multi-dimensions. A correct way is to count the number of connected components of the zero level set. However, the number of connected component does not decrease in general in multi-dimensions. For example, let $v$ be another solution with an initial value $v^0$ and consider the number of connected components of the set $A(t):=\{x\in\bfR^n:v(x,t)-u(x,t)\ge0\}$. Unfortunately, the number of connected components may increase depending on the initial distributions and the situation is far more delicate. Hence, an extension of the lap number theory or the zero set theory to multi-dimensions should be a one classifying cases when the number of connected components of the level set decreases.

The case of this paper is when $v(x,t)=\rhomx(x,t)$ with $x_0\in\bfR^n$ and $m>0$, i.e., the zero level set is
\begin{equation}\label{LevelSetRn}
A(t;m,x_0):=\{x\in\bfR^n:\rhom(x-x_0,t)-u(x,t)\ge0\}.
\end{equation}
Therefore, our chance to extend Theorem \ref{thm.generalization} to multi-dimensions comes from the fact that $\rhom(x,t)$ has a special initial value, the delta distribution, which is the steepest one. If one can show that this set is simply connected, then it may indicate that the fundamental solution $\rhom$ is steeper than any other solution. In the following theorem we will show that the zero level set $A(t;m,x_0)$ is convex for the heat equation case.

Let $u(x,t)$ be the bounded nonnegative solution of the heat equation
\begin{equation}\label{HeatEqn}
\partial_t u=\Delta u,\quad u(x,0)=u^0(x)\ge0,\quad x\in\bfR^n,\ t>0.
\end{equation}
Let $\rhom(x,t)$ be the fundamental solution of the heat equation of mass $m>0$, i.e.,
$$
\rhom(x,t)=m\phi(x,t),\quad \phi(x,t)={1\over\sqrt{4\pi t}^{\,n}}e^{-|x|^2/4t},
$$
where $\phi(x,t)$ is called the heat kernel. Then, the solution $u(x,t)$ is given by
$$
u(x,t)=u^0*\phi(t)=\int u^0(y)\phi(x-y,t)dy.
$$

\begin{theorem}\label{thm.HeatEqn} Let $u(x,t)$ be the bounded solution of the heat equation (\ref{HeatEqn}) and $\rhom(x,t)$ be the fundamental solution of mass $m>0$. Then the set $A(t;m,x_0)$ in (\ref{LevelSetRn}) is convex or empty for all $m,t>0$ and $x_0\in\bfR$.
\end{theorem}
\begin{proof} Since the heat equation is autonomous with respect to the space variable $x$, it is enough to consider the case $x_0=0$. First rewrite the level set $A$ as
$$
A(t;m)=\{x\in\bfR^n:\psi(x,t)\le1\},
$$
where $\psi(x,t):={u(x,t)\over\rhom(x,t)}$ is well-defined for all $t>0$. Rewrite $\psi(x,t)$ as
$$
\psi(x,t)=\int {u^0(y)\over m} {\phi(x-y,t)\over\phi(x,t)}dy =\int {u^0(y)\over m} e^{2x\cdot y\over4t}e^{-|y|^2\over4t}dy.
$$
Differentiating $\psi$ twice with respect to $x_i$ gives
$$
{\partial^2\over\partial x_i^2}\psi(x,t)= \int {u^0(y,t)\over m} \Big({y_i\over2t}\Big)^2e^{2x\cdot y\over4t}e^{-|y|^2\over4t}dy\ge0.
$$
Therefore, $\psi$ is convex on a line segment which is parallel to the coordinate system. Note that the heat equation is invariant under a rotation and hence $\psi$ is convex along any line segment. Suppose that the zero level set $A(t;m)$ is not convex. Then there exists $x_1,x_2\in A(t;m)$ such that $(x_1+x_2)/2\notin A(t;m)$, which contradicts to the fact that $\psi$ is convex on the line that connects $x_1$ and $x_2$. Hence the set $A(t;m)$ is convex. $\hfill\qed$
\end{proof}

This theorem gives us a hope to extend the one dimensional theory to multi-dimensions under the parabolicity assumption (\ref{PositiveDefinit}).

\section*{Acknowledgement}

The author would like to thank Lawrence C. Evans, Hirosh Manato and Athanasios Tzavaras. L.C. Evans suggested him to consider the problem in the generality of (\ref{EQN}), H. Matano gave his opinion about extending the lap number theory to multi-dimensions, and A. Tzavaras pointed out the importance of obtaining TV-boundedness without convexity assumption.
\providecommand{\bysame}{\leavevmode\hbox to3em{\hrulefill}\thinspace}
\providecommand{\MR}{\relax\ifhmode\unskip\space\fi MR }
% \MRhref is called by the amsart/book/proc definition of \MR.
\providecommand{\MRhref}[2]{%
  \href{http://www.ams.org/mathscinet-getitem?mr=#1}{#2}
}
\providecommand{\href}[2]{#2}


\begin{thebibliography}{10}

\bibitem{MR953678}
Sigurd Angenent, \emph{The zero set of a solution of a parabolic equation}, J.
  Reine Angew. Math. \textbf{390} (1988), 79--96. \MR{953678 (89j:35015)}

\bibitem{MR524760}
Donald~G. Aronson and Philippe B{\'e}nilan, \emph{R\'egularit\'e des solutions
  de l'\'equation des milieux poreux dans {${\bf R}^{N}$}}, C. R. Acad. Sci.
  Paris S\'er. A-B \textbf{288} (1979), no.~2, A103--A105. \MR{524760
  (82i:35090)}

\bibitem{MR818862}
Kuo~Shung Cheng, \emph{A regularity theorem for a nonconvex scalar conservation
  law}, J. Differential Equations \textbf{61} (1986), no.~1, 79--127.
  \MR{818862 (88e:35121)}

\bibitem{MR0481581}
C.~M. Dafermos, \emph{Characteristics in hyperbolic conservation laws. {A}
  study of the structure and the asymptotic behaviour of solutions}, Nonlinear
  analysis and mechanics: {H}eriot-{W}att {S}ymposium ({E}dinburgh, 1976),
  {V}ol. {I}, Pitman, London, 1977, pp.~1--58. Res. Notes in Math., No. 17.
  \MR{0481581 (58 \#1693)}

\bibitem{MR2574377}
Constantine~M. Dafermos, \emph{Hyperbolic conservation laws in continuum
  physics}, third ed., Grundlehren der Mathematischen Wissenschaften
  [Fundamental Principles of Mathematical Sciences], vol. 325, Springer-Verlag,
  Berlin, 2010. \MR{2574377 (2011i:35150)}

\bibitem{MR2405854}
Olivier Glass, \emph{An extension of {O}leinik's inequality for general 1{D}
  scalar conservation laws}, J. Hyperbolic Differ. Equ. \textbf{5} (2008),
  no.~1, 113--165. \MR{2405854 (2009c:35292)}

\bibitem{HaKim}
Y.~Ha and Y.-J. Kim, \emph{Fundamental solutions of a conservation law without
  convexity}, preprint, http://amath.kaist.ac.kr/papers/Kim/16.pdf.

\bibitem{MR688972}
David Hoff, \emph{The sharp form of {O}le\u\i nik's entropy condition in
  several space variables}, Trans. Amer. Math. Soc. \textbf{276} (1983), no.~2,
  707--714. \MR{688972 (84b:35080)}

\bibitem{MR1855004}
Helge~Kristian Jenssen and Carlo Sinestrari, \emph{On the spreading of
  characteristics for non-convex conservation laws}, Proc. Roy. Soc. Edinburgh
  Sect. A \textbf{131} (2001), no.~4, 909--925. \MR{1855004 (2002h:35179)}

\bibitem{KimLee}
Y.-J. Kim and Y.-R. Lee, \emph{Structure of fundamental solutions of a
  conservation law without convexity}, preprint,
  http://amath.kaist.ac.kr/papers/Kim/17.pdf.

\bibitem{MR2119939}
Philippe~G. Lefloch and Konstantina Trivisa, \emph{Continuous {G}limm-type
  functionals and spreading of rarefaction waves}, Commun. Math. Sci.
  \textbf{2} (2004), no.~2, 213--236. \MR{2119939 (2005i:35174)}

\bibitem{MR672070}
Hiroshi Matano, \emph{Nonincrease of the lap-number of a solution for a
  one-dimensional semilinear parabolic equation}, J. Fac. Sci. Univ. Tokyo
  Sect. IA Math. \textbf{29} (1982), no.~2, 401--441. \MR{672070 (84m:35060)}

\bibitem{MR0094541}
O.~A. Ole{\u\i}nik, \emph{Discontinuous solutions of non-linear differential
  equations}, Uspehi Mat. Nauk (N.S.) \textbf{12} (1957), no.~3(75), 3--73.
  \MR{0094541 (20 \#1055)}

\bibitem{Sturm}
C.~Sturm, \emph{Memoire sur une classe d'equations a differences partielles},
  J. Math. Pures Appl. \textbf{1} (1836), 373--444.

\end{thebibliography}
\end{document}